\documentclass[11pt, twoside]{article}

\usepackage{amsmath}
\usepackage{amssymb}
\usepackage{mathrsfs}
\usepackage{amsthm}
\usepackage{indentfirst}
\usepackage{color}
\usepackage{txfonts}
\usepackage{tikz}
\usetikzlibrary{decorations.pathreplacing}

\allowdisplaybreaks

\def\red{\color{red}}
\def\blue{\color{blue}}

\def\rr{\mathbb{R}}
\def\rn{\mathbb{R}^n}

\def\nn{\mathbb{N}}
\def\zz{\mathbb{Z}}
\def\on{\mathbf{0}}

\def\RM{RM_{p,q,\alpha}}
\def\one{\mathbf{1}}
\def\pmq{\frac{1}{p}-\frac{1}{q}}
\def\pa{\frac{p}{1-p\alpha}}
\def\Lq{L^q}
\def\Lpa{L^\pa}
\def\onepqa{1-p\alpha-\frac{p}{q}}

\def\lz{\lambda}

\def\dz{\delta}

\def\ez{\epsilon}

\def\bz{\beta}
\def\gz{{\gamma}}

\def\tz{\theta}

\def\ls{\lesssim}

\def\fz{\infty}
\def\az{\alpha}

\def\cf{{\mathcal F}}
\def\cg{{\mathcal G}}

\def\cl{{\mathcal L}}

\def\cx{{\mathcal X}}

\def\r{\right}
\def\lf{\left}

\def\noz{{\nonumber}}

\def\r{\right}
\def\lf{\left}

\def\dist{{\mathop\mathrm{\,dist\,}}}

\def\loc{{\mathop\mathrm{\,loc\,}}}

\def\eqref#1{(\ref{#1})}

\newtheorem{theorem}{Theorem}[section]
\newtheorem{lemma}[theorem]{Lemma}
\newtheorem{corollary}[theorem]{Corollary}
\newtheorem{proposition}[theorem]{Proposition}

\theoremstyle{definition}
\newtheorem{remark}[theorem]{Remark}

\newtheorem{question}[theorem]{Question}

\numberwithin{equation}{section}

\begin{document}

\title{\bf\Large Nontriviality of Riesz--Morrey Spaces
\footnotetext{\hspace{-0.35cm} 2020 {\it
Mathematics Subject Classification}. Primary 42B35; Secondary 46E30, 46E35. \endgraf
{\it Key words and phrases}. Lebesgue space, Morrey space, Riesz norm,
Riesz--Morrey space.
\endgraf
This project is supported by the National
Natural Science Foundation of China  (Grant Nos.  11971058 and 12071197) and
the National Key Research and Development Program of China
(Grant No. 2020YFA0712900),
and Der-Chen Chang is supported by an NSF grant DMS-1408839 and a
McDevitt Endowment Fund at Georgetown University.}}
\date{ }
\author{Zongze Zeng, Der-Chen Chang, Jin Tao and Dachun Yang\footnote{Corresponding
		author, E-mail: \texttt{dcyang@bnu.edu.cn}/{\red May 13, 2021}/Final version.}}
\maketitle

\vspace{-0.7cm}

\begin{center}
\begin{minipage}{13cm}
{\small {\bf Abstract}\quad
In this article, the authors completely answer an open question,
presented in [Banach J. Math. Anal. 15 (2021), no. 1, 20],
via showing that the Riesz--Morrey space is truly a new space
larger than a particular Lebesgue space with critical index.
Indeed, this Lebesgue space is just the real interpolation space
of the Riesz--Morrey space for suitable indices.
Moreover, the authors further show the aforementioned inclusion
is also proper, namely, this embedding is sharp in some sense,
via constructing two nontrivial spare functions,
respectively, on $\mathbb{R}^n$ and any given cube $Q_0$
of $\mathbb{R}^n$ with finite side length.
The latter constructed function is inspired by the striking function constructed by
Dafni et al. [J. Funct. Anal. 275 (2018), 577--603].
All the proofs of these results strongly depend on some
exquisite geometrical analysis on cubes of $\mathbb{R}^n$.
As an application, the relationship between Riesz--Morrey spaces
and Lebesgue spaces is completely clarified on all indices.}
\end{minipage}
\end{center}

\vspace{0.1cm}

%\tableofcontents

\vspace{0.1cm}

\section{Introduction}

Throughout the whole article, a \emph{cube $Q$} means
that it has finite side length and
all its sides parallel to the coordinate axes,
but $Q$ is not necessary to be open or closed.
Moreover, we always let $\mathcal{X}$ be $\rn$
or any given cube of $\rn$.

Recall that the \emph{Lebesgue space} $L^q(\cx)$
with $q\in[1,\fz]$ is defined to be the set of
all measurable functions $f$ on $\cx$ such that
$$\|f\|_{L^q(\mathcal{X})}:=
\begin{cases}
\displaystyle{\lf[\int_{\cx}|f(x)|^q\,dx\right]^\frac1q}
&{\rm if}\quad q\in [1,\fz),\\
\displaystyle{\mathop{\mathrm{ess\,sup}}_{x\in\cx}|f(x)|}
&{\rm if}\quad q=\fz
\end{cases}$$
is finite.
In what follows, we use $\mathbf{1}_E$ to denote
the \emph{characteristic function} of any set $E\subset\rn$,
and $L_{\loc}^q(\cx)$ to denote
the set of all measurable functions $f$ on $\cx$ such that
$f{\mathbf 1}_E\in L^q(\cx)$
for any bounded measurable set $E\subset\mathcal{X}$.
Moreover, for any $f\in L_{\loc}^1(\cx)$ and any cube $Q\subset\cx$,
let
$$f_Q:=\fint_Q f(y)\,dy:=\frac{1}{|Q|}\int_Q f(y)\,dy.$$
A surprising formula of Riesz \cite{r1910} says that,
for any given $p\in(1,\fz)$ and any given cube $Q_0$,
$f\in L^p(Q_0)$ if and only if
\begin{align}\label{riesz}
\|f\|_{R_p(Q_0)}:=\sup \lf\{ \sum_i|Q_i|\lf[
\fint_{Q_i}|f(x)|\,dx\r]^p \right\}^\frac1p<\infty,
\end{align}
where the supremum is taken over all collections of
subcubes $\{Q_i\}_i$ of $Q_0$ with pairwise disjoint interiors.
Indeed, the norm $\|\cdot\|_{ R_p(Q_0)}$ appearing in \eqref{riesz}
is nowadays called the \emph{Riesz norm}, and
$$\|\cdot\|_{R_p(Q_0)}=\|\cdot\|_{L^p(Q_0)}$$
for any given $p\in(1,\fz)$;
see also \cite[Proposition 1]{tyy21} for this identity.

To study the regularity of the solutions of
partial differential equations,
Morrey \cite{m38} introduced the \emph{Morrey space}
$$M_{q,\az}(\cx):=\lf\{f\in L^q_\loc(\cx):\,\,
\|f\|_{ M_{q,\az}(\cx)}<\fz\right\}$$
with $q\in[1,\fz)$ and $\az\in[-\frac1q,0]$,
where the \emph{Morrey norm}
\begin{align}\label{morrey}
\|f\|_{ M_{q,\az}(\cx)}:=\sup_{{\rm cube}\ Q\subset\cx}
|Q|^{-\az-\frac1q}\|f\|_{L^q(Q)}
\end{align}
with the supremum taken over all cubes of $\cx$.
Now, Morrey spaces have proved very useful function spaces
in partial differential equations and harmonic analysis.
Indeed, there exist a tremendous amount of researches
on Morrey spaces, and we refer the reader to, for instance,
the recent monographs, respectively, by Yuan et al. \cite{wsy10},
Adams \cite{a15}, and Sawano et al. \cite{sfk20i,sfk20ii}.
Also, one can find, in \cite{dn20,dr93,ly13,lwyy19,s03,tyy20},
their applications in partial differential equations
and, in \cite{ax12,hns17,hs20,ms19,mst18,tyy19},
their applications in harmonic analysis.

Very recently, via combining the Riesz norm \eqref{riesz}
and the Morrey norm \eqref{morrey},
Tao et al. \cite{tyy21} introduced the \emph{Riesz--Morrey space}
$$RM_{p,q,\alpha}(\cx):=\lf\{f\in L^q_{\loc}(\cx):\,\,
\|f\|_{RM_{p,q,\alpha}(\cx)}<\fz\right\}$$
with $p,\ q\in[1,\fz]$ and $\az\in\rr$,
where
\begin{align*}
\|f\|_{  RM_{p,q,\alpha}(\cx)}:=
\begin{cases}
\displaystyle{\sup\lf[\sum_i|Q_i|^{1-p\az-\frac pq}\|f\|_{L^q(Q_i)}^p\right]^{\frac 1p}}
&{\rm if\ }p\in[1,\fz),\ q\in[1,\fz],\\
\displaystyle{\sup_{{\rm cube\ }Q\subset\cx}|Q|^{-\alpha-\frac1q}\|f\|_{L^q(Q)}}
&{\rm if\ }p=\fz,\ q\in[1,\fz]
\end{cases}
\end{align*}
and the first supremum is taken over all collections of
subcubes $\{Q_i\}_i$ of $\cx$ with pairwise disjoint interiors.
It was shown in \cite[Theorem 1 and Corollary 1]{tyy21} that,
for a great deal of $p,\ q$, and $\az$,
the space $RM_{p,q,\alpha}(\cx)$ coincides with the (almost everywhere)
zero space or the Lebesgue space or the Morrey space.
However, there still exist three \emph{unclear} cases proposed in
\cite[Remark 3]{tyy21}, and we restate it as follows.
\begin{question}\label{open}
It will be of great interest to find some functions
which belong to the following three new ``intermediate'' space,
but not to some Lebesgue or Morrey spaces:
\begin{itemize}
\item[{\rm(i)}] $ RM_{p,q,\az}(\rn)$ for any given $p\in(1,\fz)$,
$q\in[1,p)$, and $\az\in(\frac1p-\frac1q,0)$;
	
\item[{\rm(ii)}] $ RM_{p,q,\az}(Q_0)$ for any given $p\in[1,\fz)$,
$q\in[1,p]$, and $\az\in(-\frac1q,0)$;
	
\item[{\rm(iii)}] $RM_{p,q,\az}(Q_0)$ for any given $p\in[1,\fz)$,
$q\in(p,\fz]$, and $\az\in(0,\frac1p-\frac1q)$.
\end{itemize}
(Indeed, this question was asked by the referee of \cite{tyy21}.)
\end{question}

In this article, we completely answer this open question
via showing that the Riesz--Morrey space is truly a new space
larger than a particular Lebesgue space with critical index.
Indeed, this Lebesgue space is just the real interpolation space
of the Riesz--Morrey space for suitable indices.
Moreover, we further show the aforementioned inclusion is also proper,
namely, this embedding is sharp in some sense,
via constructing two nontrivial spare functions,
respectively, on $\mathbb{R}^n$ and any given cube $Q_0$
of $\mathbb{R}^n$ with finite side length.
The latter constructed function is inspired by the striking function constructed by
Dafni et al. \cite{dhky18}. All the proofs of these results strongly depend on some
exquisite geometrical analysis on cubes of $\mathbb{R}^n$.
As an application, the relationship between Riesz--Morrey spaces
and Lebesgue spaces is completely clarified on all indices.

To be precise, we first show that (ii) and (iii) of Question \ref{open}
are partially trivial in the following proposition.

\begin{theorem}\label{thm-Q23}
Let $Q_0$ be any cube of $\rn$.
\begin{itemize}
\item[{\rm(i)}]
If $p\in[1,\fz)$, $q\in[1,p]$, and $\az\in(-\frac1q,\frac1p-\frac1q]$,
then $RM_{p,q,\az}(Q_0)=L^q(Q_0)$ and
$$\|\cdot\|_{RM_{p,q,\az}(Q_0)}
=|Q_0|^{\frac1p-\frac1q-\az}\|\cdot\|_{L^q(Q_0)}.$$

\item[{\rm(ii)}]
If $p\in[1,\fz)$, $q\in(p,\fz]$, and $\az\in(0,\frac1p-\frac1q)$,
then $RM_{p,q,\az}(Q_0)=\{{0}\}$.
\end{itemize}
\end{theorem}

Next, we consider whether or not the Riesz--Morrey space is
truly ``new'' space for the remaining case, namely,
the case when $p\in[1,\fz)$, $q\in[1,p)$,
and $\az\in(\frac1p-\frac1q,0)$.
By \cite[Theorem 1]{tyy21} and Theorem \ref{thm-Q23},
we have
$$RM_{p,q,0}(\cx)=L^p(\cx)\quad {\rm and}
\quad RM_{p,q,\frac1p-\frac1q}(\cx)=L^q(\cx).$$
Moreover, recall that the \emph{real interpolation space}
$(L^p(\cx),L^q(\cx))_t$,
between $L^p(\cx)$ and $L^q(\cx)$, is $L^\tz(\cx)$,
where $t\in(0,1)$, and $\tz$ satisfies
$$\frac1\tz:=\frac{1-t}{p}+\frac{t}{q}.$$
Replacing $t$ by $\az/(\frac1p-\frac1q)$,
we then have
$$\tz=\frac{p}{1-p\az}$$
with $\az\in(\frac1p-\frac1q,0)$.
Therefore, it is natural to ask whether or not
$L^{\frac{p}{1-p\az}}(\cx)=RM_{p,q,\az}(\cx)$ holds true
for any given $p\in[1,\fz)$, $q\in[1,p)$,
and $\az\in(\frac1p-\frac1q,0)$.
Indeed, we give a negative answer to this question
in the following theorem, which shows that
Riesz--Morrey spaces are more wider than Lebesgue spaces.
\begin{theorem}\label{thm-Q12}
Let $p\in(1,\fz)$, $q\in[1,p)$, and $\az\in(\frac1p-\frac1q,0)$.
\begin{itemize}
\item[{\rm(i)}]
If $L^{\tz}(\rn)\subset RM_{p,q,\az}(\rn)$, then $\tz=\frac{p}{1-p\az}$.
Moreover, if $L^\tz(Q_0)\subset RM_{p,q,\az}(Q_0)$,
with $Q_0$ be any given cube of $\rn$,
then $\tz\in[\frac{p}{1-p\az},\fz]$.

\item[{\rm(ii)}]
The index $\frac{p}{1-p\az}$ in ${\rm (i)}$ is sharp, namely,
$L^{\frac{p}{1-p\az}}(\cx)\subsetneqq RM_{p,q,\az}(\cx)$.
\end{itemize}
\end{theorem}
\begin{remark}
\begin{enumerate}
\item[{\rm(i)}]
From Theorem \ref{thm-Q12}(i), it follows that
$\tz=\frac{p}{1-p\az}$ is the only possible $\tz$
such that $L^{\tz}(\rn)\subset RM_{p,q,\az}(\rn)$,
and also the minimal $\tz$
(corresponding to the largest Lebesgue space over the cube $Q_0$)
such that $L^{\tz}(Q_0)\subset RM_{p,q,\az}(Q_0)$.
Meanwhile, Theorem \ref{thm-Q12}(ii) further shows that
the above embedding is proper.
Thus, in this sense, the index $\frac{p}{1-p\az}$ is sharp.

\item[{\rm(ii)}]
Question \ref{open}(i) is answered in Theorem \ref{thm-Q12}.
Question \ref{open}(ii) with $\az\in(-\frac1q,\frac1p-\frac1q]$
is answered in Theorem \ref{thm-Q23}(i).
Question \ref{open}(ii) with $\az\in(\frac1p-\frac1q,0)$
is answered in Theorem \ref{thm-Q12}.
Question \ref{open}(iii) is answered in Theorem \ref{thm-Q23}(ii).
To sum up, Question \ref{open} is completely answered in
Theorems \ref{thm-Q23} and \ref{thm-Q12}.
\end{enumerate}
\end{remark}
Let us briefly describe some features of functions
in this new space over $\rr$ or $I_0:=(0,1)$;
rigorous constructions and calculations
on high dimension are given later in Section \ref{sec3}.
The embedding in Theorem \ref{thm-Q12}(ii)
is established in Proposition \ref{prop-<} below.
Toward the proper inclusion in Theorem \ref{thm-Q12}(ii),
we modify some non-integrable (with power $\frac{p}{1-p\az}$)
functions over $\rr$ or $I_0$,
respectively, to make it sparse. To be precise,
\begin{enumerate}
\item[{\rm(i)}] on $\rr$, we consider the function
$h(x):=1$ for any $x\in\rr$.
Obviously, $h\notin L^{\frac{p}{1-p\az}}(\rr)$
and $h\notin RM_{p,q,\az}(\rr)$.
Thus, we choose a neighborhood of $\fz$
and divide it into countable disjoint subintervals,
then the new obtained function is also not integrable
by the translation invariance of the Lebesgue integral,
but such a function belongs to the Riesz--Morrey space
so long as the partition is sparse enough;
see the \emph{exponential} gaps of the example
in the proof of Proposition \ref{prop-rn} below;

\item[{\rm(ii)}] on $I_0$, we consider the function
$g(x):=x^{\frac1p-\az}$ for any $x\in I_0$.
Apparently, $g\notin L^{\frac{p}{1-p\az}}(I_0)$.
Moreover, $g\notin RM_{p,q,\az}(I_0)$;
see \eqref{f1} below for the exact proof.
In this case, the singularity of $g$ is $0$, and the above method
on $\rr$ is no longer feasible because $|\rr|=\fz$, but $|I_0|<\fz$.
To obtain the desired function on $I_0$, we borrow some ideas from
the function constructed by Dafni et al. in \cite[Proposition 3.2]{dhky18},
which provides a sparse version of $g$ near $0$ and keeps its
integral infinity; see the proof of Proposition \ref{prop-Q} below.
\end{enumerate}

The organization of the remainder of this article is as follows.

Section \ref{sec2}  is devoted to the proof of Theorem \ref{thm-Q23}.
We first prove Theorem \ref{thm-Q23}(i).
Via establishing Lemma \ref{lem-1E} below,
we then prove Theorem \ref{thm-Q23}(ii).

In Section \ref{sec3}, we first prove Theorem \ref{thm-Q12}(i).
Next, we divide the proof of Theorem \ref{thm-Q12}(ii) into three parts,
namely, Propositions \ref{prop-<}, \ref{prop-rn}, and \ref{prop-Q}.
Proposition \ref{prop-<} shows the embedding
$L^{\frac{p}{1-p\az}}(\cx)\subset RM_{p,q,\az}(\cx)$.
Proposition \ref{prop-rn} is devoted to the proper inclusion
on $\rn$, the main idea of which is to construct a function
based on a sparse family $\{P_\ell\}_{\ell\in\nn}$ of cubes
with exponential gaps.
When estimating the Riesz--Morrey norm,
we need to consider any given collection of cubes
$\{Q_i\}_i$ with pairwise disjoint interiors.
It is easy to calculate the case when $Q\in \{Q_i\}_i$ is \emph{small},
that is, $Q$ intersects no more than one element in $\{P_\ell\}_{\ell\in\nn}$.
However, when $Q\in \{Q_i\}_i$ is \emph{large},
that is, $Q$ intersects no less than two elements in $\{P_\ell\}_{\ell\in\nn}$,
the estimates become complicated.
To surmount this, we rearrange these ``large'' cubes via a monotone sequence;
see \eqref{mono} below.
Combining this and the exponential gaps of $\{P_\ell\}_{\ell\in\nn}$,
we obtain a nice bounded dominating function, namely, \eqref{g-domin} below,
which completes the proof of Proposition \ref{prop-rn}.
Proposition \ref{prop-Q} is devoted to the proper inclusion
on $Q_0$, which is more tough than that on $\rn$.
To do this, we introduce a function with three series of parameters
(namely, lengths $\{l_i\}_{i=0}^\fz$, distances $\{d_i\}_{i=0}^\fz$,
and heights $\{h_i\}_{i=0}^\fz$) as in the proof of
\cite[Proposition 3.2]{dhky18}.
Since the Riesz--Morrey norm has no mean oscillation,
we need to choose different parameters comparing with
the proof of \cite[Proposition 3.2]{dhky18},
and also estimate the Riesz--Morrey norm via different methods.
Moreover, the changement of parameters brings an essential obstacle,
that is, the family of cubes we constructed may
no longer mutually disjoint.
To overcome it, we use a technique of rearrangement
which did not appear in \cite{dhky18} before;
see (b)$_1$ through (b)$_3$ in the proof of Proposition \ref{prop-Q} below.
As an application, we summarize all the classifications
of the Riesz--Morrey space in Corollary \ref{coro} below,
which is now completely clarified on all indices.

Below we make some conventions on notation.
The origin of $\rn$ is denoted by $\mathbf{0}$.
Let $\nn:=\{1,2,\ldots\}$ and $\zz_+:=\{0,1,2,\ldots\}$.
For any set $F$, $\# F$ denotes its \emph{cardinality}.
Let $E$ be a subset of $\rn$,
we denote by $\mathbf{1}_E$ its \emph{characteristic function}.
For any interval $I$ of $\rr$,
$I^n:=I\times\cdots\times I$
denotes a cube in $\rn$.
We use $C$ to denote a positive constant
which is independent of the main parameters,
but it may vary from line to line.
Constants with subscripts, such as $C_0$ and $A_1$, do
not change in different occurrences.
The symbol $f\lesssim g$ represents that $f\leq Cg$ for some positive
constant $C$. We write $f\sim  g$ if $f\lesssim g$ and $g\lesssim f$.
If $f\leq Cg$ and $g=h$ or $g\leq h$, we then write $f\lesssim g\sim h$ or $f\lesssim g\lesssim h$
instead of $f\lesssim g=h$ or $f\lesssim g\leq h$.

\section{Proof of Theorem \ref{thm-Q23}}\label{sec2}

This section is devoted to the proof of Theorem \ref{thm-Q23}.
We first recall the following basic inequalities
and then give the proof of Theorem \ref{thm-Q23}(i).

\begin{lemma}\label{lem-eq}
Let $\{a_j\}_{j\in\nn}$ be any sequence of positive numbers.
\begin{enumerate}
\item[{\rm(i)}]
If $\gz\in[1,\fz)$, then
$\sum_{j\in\nn} a_j^\gz
\le (\sum_{j\in\nn} a_j)^\gz$.

\item[{\rm(ii)}]
If $\gz\in[0,1]$, then, for any $N\in\nn$,
$\sum_{j=1}^N a_j^\gz
\le N^{1-\gz}(\sum_{j=1}^N a_j)^\gz$.

\item[{\rm(iii)}]
If $\gz\in[0,1]$, then
$\sum_{j\in\nn} a_j^\gz
\ge (\sum_{j\in\nn} a_j)^\gz$.

\item[{\rm(iv)}]
If $\gz\in[1,\fz)$, then, for any $N\in\nn$,
$\sum_{j=1}^N a_j^\gz
\ge N^{1-\gz}(\sum_{j=1}^N a_j)^\gz$.
\end{enumerate}
\end{lemma}

Lemma \ref{lem-eq} is well known
and hence we omit its proof here;
see, for instance, \cite[p.\,12, Exercise 1.1.4]{gtm249}.

\begin{proof}[Proof of Theorem \ref{thm-Q23}(i)]
Let $p\in [1,\infty),$ $q\in [1, p]$,
$\alpha\in (-\frac{1}{q}, \pmq]$, and $f\in L^q(Q_0)$.
Then $1-p\alpha-\frac{p}{q}\in [0, 1)$ and hence,
for any collection of subcubes $\{Q_i\}$ of $Q_0$
with pairwise disjoint interiors,
we have
\begin{align*}
&\sum_i |Q_i|^{\onepqa}\|f\|^p_{\Lq(Q_i)}\\
&\quad\le  \sum_i |Q_0|^{\onepqa}\|f\|^p_{\Lq(Q_i)}
= |Q_0|^{\onepqa} \sum_i \left[\int_{Q_i} |f(x)|^q \,dx\right]^{\frac{p}{q}}\\
&\quad\le |Q_0|^{\onepqa}\left[\sum_i\int_{Q_i} |f(x)|^q \,dx\right]^{\frac{p}{q}}
=|Q_0|^{\onepqa}\left[\int_{Q_0} |f(x)|^q \,dx\right]^{\frac{p}{q}},
\end{align*}
where we used Lemma \ref{lem-eq}(i) in the last inequality
with the observation $\frac pq\in[1,\fz)$.
This implies that
$$\|f\|_{RM_{p,q,\az}(Q_0)}
\le |Q_0|^{\frac1p-\frac1q-\az}\|f\|_{L^q(Q_0)},$$
and hence $f\in RM_{p,q,\az}(Q_0)$.
Thus, $RM_{p,q,\az}(Q_0)\supseteq \Lq(Q_0)$.

On the other hand, from the definitions of
$RM_{p,q,\az}(Q_0)$ and $\|\cdot\|_{RM_{p,q,\az}(Q_0)}$,
we deduce that
$RM_{p,q,\az}(Q_0)\subseteq \Lq(Q_0)$ and
$$\|f\|_{RM_{p,q,\az}(Q_0)}
\ge \lf\{|Q_0|^{\onepqa}\|f\|^p_{\Lq(Q_0)}\r\}^{\frac1p}
=|Q_0|^{\frac1p-\frac1q-\az}\|f\|_{L^q(Q_0)}.$$

To sum up, we have $RM_{p,q,\az}(Q_0) = \Lq(Q_0)$
and
$$\|f\|_{RM_{p,q,\az}(Q_0)}
= |Q_0|^{\frac1p-\frac1q-\az}\|f\|_{L^q(Q_0)}.$$
This finishes the proof of Theorem \ref{thm-Q23}(i).
\end{proof}

Next, we prove Theorem \ref{thm-Q23}(ii)
via first establishing the following elementary lemma.
\begin{lemma}\label{lem-1E}
Let $p\in[1,\fz)$, $q\in (p, \infty]$, $\alpha\in (0, \pmq)$,
$Q_0\subset\rn$ be any cube of $\rn$,
and $E\subset Q_0$ be any measurable set with $|E|> 0$.
Then
\[\|\one_E\|_{\RM(Q_0)} = \infty.\]
\end{lemma}
\begin{proof}
Without loss of generality, we may assume that
$Q_0 := [-1,1]^n$.
For any $t\in (0, 1]$, let $tQ_0 := [-t, t]^n$.

Since $p\az-1\in(-1,-\frac1q)\subset(-1,0)$,
it follows that $\sum_{\ell=1}^\infty \ell^{1/(p\alpha-1)}<\fz$.
Moreover, we claim that there exist $\{t_k\}_{k\in\nn}\subset(0,1)$
such that
\begin{align}\label{tk}
1=:t_0>t_1>t_2>\cdots>t_k>\cdots
\end{align}
and, for any $k\in\nn$
\begin{align}\label{k-1}
\int_{t_{k-1}Q_{0}\setminus t_{k}Q_{0}} \one_E(x) \,dx
=\frac{|E|}{2}\frac{k^{1/(p\alpha-1)}}{\sum_{\ell=1}^\infty \ell^{1/(p\alpha-1)}}.
\end{align}
Indeed, for any $t\in[0,1]$, let $g(t):=|tQ_0\cap E|$.
Then $g$ is increasing and
continuous by the continuity of the Lebesgue measure.
From this, $g(0)=0$, and $g(1)=|E|>\frac{|E|}{2}>0$,
we deduce that, for any $k\in\nn$,
$$t_k:=\inf\lf\{t\in[0,1]:\
g(t)\ge\frac{|E|}{2}
\frac{\sum_{\ell=k+1}^\infty \ell^{1/(p\alpha-1)}}
{\sum_{\ell=1}^\infty \ell^{1/(p\alpha-1)}}\r\}$$
satisfies \eqref{tk} and, moreover,
\begin{align*}
\int_{t_{k-1}Q_{0}\setminus t_{k}Q_{0}} \one_E(x) \,dx
&=|(t_{k-1}Q_{0}\setminus t_{k}Q_{0})\cap E|
=g(t_{k-1})-g(t_{k})\\
&=\frac{|E|}{2}\frac{\sum_{\ell=k}^\infty \ell^{1/(p\alpha-1)}}
{\sum_{\ell=1}^\infty \ell^{1/(p\alpha-1)}}
-\frac{|E|}{2}\frac{\sum_{\ell=k+1}^\infty \ell^{1/(p\alpha-1)}}
{\sum_{\ell=1}^\infty \ell^{1/(p\alpha-1)}}\\
&=\frac{|E|}{2}\frac{k^{1/(p\alpha-1)}}{\sum_{\ell=1}^\infty \ell^{1/(p\alpha-1)}},
\end{align*}
which shows that \eqref{k-1} holds true.
This finishes the proof of the above claim.

Now, for any given $k\in\nn$,
we can divide $t_{k-1}Q_{0}\setminus t_{k}Q_{0}$
into a family of cubes with pairwise disjoint interiors,
denoted by $\{Q_j^{(k)}\}_{j=1}^{N_k}$
with $N_k\in\nn\cup \{\fz\}$. Then
\[\mathcal{F} := \lf\{Q_j^{(k)}:\ k\in \nn\ {\rm and}\ j\in \{1,\cdots,N_k\}\r\}\]
is a  family of subcubes of $Q_0$ with pairwise disjoint interiors.
Then, by $\onepqa\in(0,1-\frac{p}{q})$, we conclude that
\begin{align*}
&\sum_{Q\in\mathcal{F}}|Q|^{\onepqa} \|\one_E\|^p_{\Lq(Q)}\\
&\quad=\sum_{k=1}^{\infty} \sum_{j=1}^{N_k} |Q_j^{(k)}|^{\onepqa} \|\one_E\|^p_{\Lq(Q_j^{(k)})}
=\sum_{k=1}^{\infty} \sum_{j=1}^{N_k} |Q_j^{(k)}|^{\onepqa} |Q_j^{(k)}\cap E|^{\frac{p}{q}}\\
&\quad\ge\sum_{k=1}^{\infty} \sum_{j=1}^{N_k} |Q_j^{(k)}\cap E|^{\onepqa}
 |Q_j^{(k)}\cap E|^{\frac{p}{q}}
=\sum_{k=1}^{\infty} \sum_{j=1}^{N_k} |Q_j^{(k)}\cap E|^{1-p\alpha}.
\end{align*}
From this, $1-p\alpha\in(\frac1q,1)\subset(0,1)$,
Lemma \ref{lem-eq}(iii), and \eqref{k-1},
it follows that
\begin{align*}
&\sum_{Q\in\mathcal{F}}|Q|^{\onepqa} \|\one_E\|^p_{\Lq(Q)}\\
&\quad\ge  \sum_{k=1}^{\infty} \sum_{j=1}^{N_k} |Q_j^{(k)}\cap E|^{1-p\alpha}
\ge  \sum_{k=1}^{\infty} \left[ \sum_{j=1}^{N_k} |Q_j^{(k)}\cap E|\right]^{1-p\alpha}\\
&\quad= \sum_{k=1}^{\infty}
|(t_{k-1}Q_{0}\setminus t_{k}Q_{0})\cap E|^{1-p\alpha}
= \left[\frac{|E|}{2\sum_{\ell=1}^\infty \ell^{1/(p\alpha-1)}}\right]^{1-p\alpha}
\sum_{k = 1}^{\infty} k^{-1}= \infty,
\end{align*}
which implies that $\|\one_E\|_{\RM(Q_0)} = \infty$.
This finishes the proof of Lemma \ref{lem-1E}.
\end{proof}

\begin{proof}[Proof of Theorem \ref{thm-Q23}(ii)]
Let $p\in[1,\fz)$, $q\in(p,\fz]$, $\az\in(0,\frac1p-\frac1q)$,
and $f\in\RM(Q_0)$.
We show that $\|f\|_{\Lq(Q_0)} = 0$.
Indeed, if $\|f\|_{\Lq(Q_0)} > 0$,
then we claim that there exist $\epsilon,\ \delta\in(0,\fz)$
such that $|E|> \delta > 0$,
where
$$E:=\{x\in Q_0:\ |f(x)|>\ez\}.$$
Otherwise, for any $\ez\in(0,\fz)$, we have
$|\{x\in Q_0:\ |f(x)|>\ez\}|=0$, and hence
$$|\{x\in Q_0:\ |f(x)|>0\}|
=\lf|\bigcup_{k=1}\lf\{x\in Q_0:\ |f(x)|>\frac1k\r\}\r|=0.$$
This implies that $\|f\|_{L^q(Q_0)}=0$
which contradicts to  $\|f\|_{L^q(Q_0)}>0$.
Therefore, the above claim holds true.
By this claim and Lemma \ref{lem-1E}, we obtain
$$\|\one_{E}\|_{\RM(Q_0)} = \infty.$$
Meanwhile, notice that, for functions $f_1$ and  $f_2$ with
$|f_1|\ge|f_2|$, by the definition of $\|\cdot\|_{\RM(Q_0)}$, we have
$$\|f_1\|_{\RM(Q_0)} \ge \|f_2\|_{\RM(Q_0)}.$$
Thus,
$$\|f\|_{\RM(Q_0)} \ge \ez\lf\|\one_{E}\r\|_{\RM(Q_0)} = \infty$$
and hence $f \notin \RM(Q_0)$
which contradicts to the fact that $f\in\RM(Q_0)$.
This shows that $\|f\|_{\Lq(Q_0)} = 0$,
and hence $\RM(Q_0) = \{0\}$,
which completes the proof of Theorem \ref{thm-Q23}(ii)
and hence of Theorem \ref{thm-Q23}.
\end{proof}

\section{Proof of Theorem \ref{thm-Q12}}\label{sec3}
In this section, we first give the proof of Theorem \ref{thm-Q12}(i),
and then prove Theorem \ref{thm-Q12}(ii) via
establishing Propositions \ref{prop-<}, \ref{prop-rn},
and \ref{prop-Q} below.

\begin{proof}[Proof of Theorem \ref{thm-Q12}(i)]
Let $p\in(1,\fz)$, $q\in[1,p)$, and $\alpha\in (\pmq, 0)$.
To prove Theorem \ref{thm-Q12}(i), we need to find two functions
$f_1$ and $f_2$ such that, for any $\tz\in[1,\frac{p}{1-p\az})$,
$$f_1\in L^\tz_{\loc}(\rn)\setminus\RM(\rn)$$
and, for any $\tz\in(\frac{p}{1-p\az},\fz]$,
$$f_2\in L^\tz(\rn)\setminus\RM(\rn).$$

Now, let $f(x):= |x|^{n(\alpha-\frac{1}{p})}$ for any $x\in (0, \infty)^n$.
Let $N\in\nn$ satisfy $N>\sqrt{n}$.
For any $i\in\zz$, let $A_i:=(0,N^i]^n$,
$B_i:= A_{i+1}\setminus A_{i}$,
and $T_i:= [B(\on,N^{i+1})\setminus B(\on,\sqrt{n}N^i)]\cap(0,\fz)^n$.
Let $$\mathbb{S}^{n-1}_+:=\lf\{x=(x_1,\,\dots,\,x_n)\in\rn:\
|x|=1,\ {\rm and}\ x_j\ge0,\ \forall\,j\in\{1,\,\dots,\,n\}\r\}.$$
Then, by some geometrical observations,
for any $i\in\zz$, we have $B_i \supseteq T_i$ and hence
\begin{align}\label{N-n}
\int_{B_i}|f(x)|^q\,dx
\ge&\int_{T_i}|f(x)|^q\,dx
=\int_{T_i}|x|^{qn\alpha-\frac{qn}{p}}\,dx\noz \\
=&\int_{\mathbb{S}^{n-1}_+}\int_{\sqrt{n}N^i}^{N^{i+1}}
r^{qn\alpha-\frac{qn}{p}}r^{n-1}\,dr\,d\sigma
= |\mathbb{S}^{n-1}_+|\int_{\sqrt{n}N^i}^{N^{i+1}}
r^{qn\alpha-\frac{qn}{p}+n-1}\,dr\noz\\
=&\frac{|\mathbb{S}^{n-1}_+|}{qn\alpha-\frac{qn}{p}+n}N^{in(q\alpha-\frac{q}{p}+1)}
\lf(N^{qn\alpha-\frac{qn}{p}+n}-\sqrt{n}^{qn\alpha-\frac{qn}{p}+n}\r).
\end{align}
Observe that, for any $i\in\zz$, the ring-like set $B_i = A_{i+1}\setminus A_i$
consists of $N^n-1$ interior pairwise disjoint subcubes
denoted, respectively, by $\{Q_j^{(i)}\}_{j=1}^{N^n-1}$,
whose side length equals to $N^i$.
Therefore, for any $i\in\zz$, we have
$B_i = \bigcup_{j=1}^{N^n-1}Q_{j}^{(i)}$.
Moreover, for any $i\in\zz$, from $\frac{p}{q}\in[1,\fz)$,
Lemma \ref{lem-eq}(iv), and \eqref{N-n},  we deduce that
\begin{align*}
&\sum_{j=1}^{N^n-1}|Q_{j}^{(i)}|^{1-p\alpha-\frac{p}{q}}\|f\|^p_{L^q(Q_j^{(i)})}\\
&\quad=N^{in(1-p\alpha-\frac{p}{q})}\sum_{j=1}^{N^n-1}\|f\|^{q\frac pq}_{L^q(Q_j^{(i)})}
\ge N^{in(1-p\alpha-\frac{p}{q})}(N^n-1)^{1-\frac{p}{q}}
\lf[\sum_{j=1}^{N^n-1}\|f\|^q_{L^q(Q_j^{(i)})}\r]^{\frac pq} \\
&\quad=(N^n-1)^{1-\frac{p}{q}}N^{in(1-p\alpha-\frac{p}{q})}
\lf[\int_{B_i}|f(x)|^q\,dx\r]^{\frac{p}{q}}\\
&\quad\ge(N^n-1)^{1-\frac{p}{q}}\lf[\frac{|\mathbb{S}^{n-1}_+|
(N^{qn\alpha-\frac{qn}{p}+n}-\sqrt{n}^{qn\alpha-\frac{qn}{p}+n})}
{qn\alpha-\frac{qn}{p}+n}\r]^{\frac{p}{q}}\\
&\quad=:C_0,
\end{align*}
where $C_0$ is a positive constant independent of $i$.
Split $f$ into $f_1 := f\one_{A_0}$ and $f_2 := f\one_{A_0^\complement}$.
Then, by some trivial calculations, we obtain
$$f_1\in L^\tz_{\loc}(\rn)\iff 1\le\tz<\frac{p}{1-p\az}$$
and
$$f_2\in L^\tz(\rn)\iff \frac{p}{1-p\az}<\tz\le\fz.$$
But, we have
\begin{align}\label{f1}
\|f_1\|_{\RM(\rn)}^p
\ge \sum_{i=-\infty}^{-1}
\sum_{j=1}^{N^n-1}|Q_j^{(i)}|^{\onepqa}\|f\|^p_{L^q(Q_j^{(i)})}
\ge\sum_{i=-\infty}^{-1}C_0= \infty
\end{align}
and
$$
\|f_2\|_{\RM(\rn)}^p
\ge \sum_{i=0}^{\fz}
\sum_{j=1}^{N^n-1}|Q_j^{(i)}|^{\onepqa}\|f\|^p_{L^q(Q_j^{(i)})}
\ge\sum_{i=0}^{\fz}C_0= \infty,
$$
which shows that $f_1\notin \RM(\rn)$ and $f_2\notin \RM(\rn)$.
Thus, for any $\tz\in [1, \infty]$ with $\tz\ne \pa$,
$$L^\tz(\rn)\nsubseteq \RM(\rn),$$
which implies that, if $L^{\tz}(\rn)\subset RM_{p,q,\az}(\rn)$,
then $\tz=\frac{p}{1-p\az}$.

Moreover, we also have $f_1\notin \RM(A_0)$
and hence,
for any $\tz\in [1, \frac{p}{1-p\az})$,
$$L^\tz(A_0)\nsubseteq \RM(A_0),$$
which further implies that, if
$L^\tz(A_0)\subset RM_{p,q,\az}(A_0)$,
then $\tz\in[\frac{p}{1-p\az},\fz]$.
Using the technique of translation and dilation,
we find that the above conclusion holds true with $A_0$
replaced by any given cube $Q_0$ of $\rn$.
This finishes the proof of Theorem \ref{thm-Q12}(i).
\end{proof}

Next, we prove Theorem \ref{thm-Q12}(ii)
via first establishing the following embedding.

\begin{proposition}\label{prop-<}
Let $p\in(1,\fz)$, $q\in[1,p)$, and $\az\in(\frac1p-\frac1q,0)$.
Then
$$L^{\frac{p}{1-p\az}}(\cx)\subset RM_{p,q,\az}(\cx)
\quad{\rm and}\quad\|\cdot\|_{RM_{p,q,\az}(\cx)}
\le \|\cdot\|_{L^{\frac{p}{1-p\az}}(\cx)}.$$
\end{proposition}
\begin{proof}
Let $p,\ q$, and $\alpha$ be as in this proposition.
For any cube $Q$ of $\rn$, by $\pa\in (q, p)$ and
the H\"older inequality, we conclude that
\begin{align}\label{holder1}
|Q|^{\onepqa}\|f\|^p_{\Lq(Q)}
&=|Q|^{\onepqa}\lf[\int_{Q}|f(x)|^q \,dx\r]^{\frac{p}{q}}
\le\lf[\int_{Q}|f(x)|^{\pa} \,dx\r]^{1-p\az}.
\end{align}
Thus, for any collection of subcubes $\{Q_i\}_i$ of $\cx$
with pairwise disjoint interiors,
by \eqref{holder1}, $1-p\az\in(1,\frac{p}{q})$,
and Lemma \ref{lem-eq}(i), we have
\begin{align*}
\sum_i |Q_i|^{\onepqa}\|f\|^p_{\Lq(Q_i)}
&\le \sum_i \lf[\int_{Q_i}|f(x)|^{\pa} \,dx\r]^{1-p\az}\\
&\le \lf[\sum_i \int_{Q_i}|f(x)|^{\pa} \,dx\r]^{1-p\az}
\le \lf[ \int_{\cx}|f(x)|^{\pa} \,dx\r]^{1-p\alpha},
\end{align*}
which implies that
$$\|f\|_{RM_{p,q,\az}(\cx)}
\le \|f\|_{L^{\frac{p}{1-p\az}}(\cx)}$$
and hence $L^{\frac{p}{1-p\az}}(\cx)\subset RM_{p,q,\az}(\cx)$.
This finishes the proof of Proposition \ref{prop-<}.
\end{proof}

\begin{remark}\label{rem-q}
In Proposition \ref{prop-<}, the index of the Lebesgue space, namely,
$\frac{p}{1-p\az}$, is independent of $q$.
Moreover, if we fix $p$ and $\az$, then Riesz--Morrey spaces are
monotonic with respect to $q$.  To be precise,
let $p\in (1, \infty), q\in [1, p)$, and $\alpha\in (\pmq, 0)$.
Then, for any $\beta\in (q, p)$, we claim that
\[RM_{p, \beta, \alpha}(\cx)\subseteq \RM(\cx).\]
Indeed, for any given collection of subcubes $\{Q_i\}$ of $\cx$
with pairwise disjoint interiors,
by $\beta > q$ and the H\"older inequality, we conclude that
\[
\begin{aligned}
\sum_i |Q_i|^{\onepqa}\|f\|^p_{\Lq(Q_i)}
&=\sum_i |Q_i|^{1-p\alpha-\frac{p}{q}} \lf[\int_{Q_i}|f(x)|^q \,dx\r]^{\frac{p}{q}}\\
&\le\sum_i |Q_i|^{1-\frac{p}{\bz}-p\alpha} \lf[\int_{Q_i}|f(x)|^\bz \,dx\r]^{\frac{p}{\bz}}
\end{aligned}
\]
and hence $\|f\|_{RM_{p, \beta, \alpha}(\cx)} \ge \|f\|_{\RM(\cx)}$.
Thus, $RM_{p, \beta, \alpha}(\cx)\subseteq \RM(\cx)$,
which shows that the above claim holds true.
\end{remark}

Now, we show the proper inclusion on $\rn$
via constructing a sparse function
associated with a family of cubes
with exponential gaps.

\begin{proposition}\label{prop-rn}
Let $p\in(1,\fz)$, $q\in[1,p)$, and $\az\in(\frac1p-\frac1q,0)$.
Then there exists an $f\in RM_{p,q,\az}(\rn)\setminus L^{\frac{p}{1-p\az}}(\rn)$.
\end{proposition}
\begin{proof}
Let $p,\ q$, and $\az$ be as in this proposition.
For any $\ell\in\nn$, let $P_\ell:= [2^\ell,2^\ell+\ell^{-1/n}]^n$.
Then $\{P_\ell\}_{\ell\in\nn}$ is obviously a collection of disjoint subcubes of $\rn$.
Let $$f:=\sum_{\ell\in\nn}\one_{P_\ell}.$$
Then $f$ is bounded and hence $f\in L^q_{\textrm{loc}}(\rn)$.
Moreover, since
$$\|f\|_{\Lpa(\rn)}^{\pa}=\sum_{\ell\in\nn}|P_\ell|
=\sum_{\ell\in\nn}\frac1\ell=\fz,$$
we have $f\notin \Lpa(\rn)$.
Thus, to prove this proposition, it remains to show that $f\in\RM(\rn)$.

In what follows, we use $Q^\circ$ to denote the \emph{interior} of $Q$.
For any given collection of subcubes $\cf:=\{Q_i\}_i$ of $\rn$
with pairwise disjoint interiors,
we split it into three subcollections as follows:
\begin{itemize}
\item[{\rm(i)}]
$\cf_1:=\lf\{Q\in\cf:\
Q^\circ\ {\rm intersects}\ \ge2{\rm\ elements\ of\ }
\{P_\ell\}_{\ell\in \nn}\r\}$,

\item[{\rm(ii)}]
$\cf_2:=\lf\{Q\in\cf:\
Q^\circ\ {\rm intersects}\ 1 {\rm\ element\ of\ }
\{P_\ell\}_{\ell\in \nn}\r\}$,

\item[{\rm(iii)}]
$\cf_3:=\lf\{Q\in\cf:\
Q^\circ\ {\rm intersects\ no\ element\ of\ }
\{P_\ell\}_{\ell\in \nn}\r\}$.
\end{itemize}
From this and the definition of $f$, it follows that
\begin{align}\label{I1I2}
&\sum_{Q\in\cf} |Q|^{\onepqa}\|f\|^p_{\Lq(Q)}\noz\\
&\quad= \sum_{Q\in\cf_1} |Q|^{\onepqa}\|f\|^p_{\Lq(Q)}
+\sum_{Q\in\cf_2} \cdots
+\sum_{Q\in\cf_3} \cdots\noz\\
&\quad=\sum_{Q\in\cf_1} |Q|^{\onepqa}\|f\|^p_{\Lq(Q)}+
\sum_{Q\in\cf_2} \cdots\noz\\
&\quad=:{\rm I}_1+{\rm I}_2.
\end{align}

We first estimate ${\rm I}_2$. For any $Q\in\cf_2$,
$Q^\circ$ intersects only one element
of $\{P_\ell\}_{\ell\in \nn}$,
denoted by $P_Q$.
By this, $\onepqa\in(1-\frac{p}{q},0)$, $1-p\az\in(1,\frac{p}{q})$,
and the disjointness of $\{{Q}_{i}^\circ\}_i$,
we conclude that
\begin{align}\label{I2}
{\rm I}_2 &=  \sum_{Q\in\cf_2}|Q|^{\onepqa}\|f\|^p_{\Lq(Q)}
\le \sum_{Q\in\cf_2} |Q\cap P_Q|^{\onepqa}
 |Q\cap P_Q|^{\frac{p}{q}}\noz\\
&=\sum_{Q\in\cf_2}|Q\cap P_Q|^{1-p\az}
\le\sum_{\ell\in\nn} \sum_{i} |{Q}_{i}\cap P_\ell|^{1-p\az}
\le \sum_{\ell\in\nn} \lf(\sum_{i} |{Q}_{i}\cap P_\ell|\r)^{1-p\az}\noz\\
&\le \sum_{\ell\in\nn} |P_\ell|^{1-p\az}
=\sum_{\ell\in\nn}\frac{1}{\ell^{1-p\az}}<\fz.
\end{align}
This is a desired estimate of ${\rm I}_2$.

Next, we estimate ${\rm I}_1$ via rearranging cubes in $\cf_1$.
To this end,
we first claim that, if cube $Q$ of $\rn$ satisfies that
\begin{align}\label{DlDm}
Q^\circ\cap P_\ell\neq\emptyset
\quad{\rm and}\quad
Q^\circ\cap P_m\neq\emptyset
\end{align}
with $1\le\ell<m$, then
\begin{align}\label{QBlm}
Q\supsetneqq B_{\ell,m},
\end{align}
here and thereafter, for any $\ell,\ m\in\nn$ with $\ell<m$,
$$B_{\ell,m}:= [2^\ell+\ell^{-1/n}, 2^m]^n.$$
Indeed, let $Q=:[a_1,b_1]\times\cdots\times[a_n,b_n]$.
Then \eqref{DlDm} shows that there exist
$x:=(x_1,\,\cdots,\,x_n)\in Q^\circ\cap P_\ell$
and $y:=(y_1,\,\cdots,\,y_n)\in Q^\circ\cap P_m$.
Therefore, for any $j\in\{1,\,\dots,\,n\}$,
$$a_j< x_j\le2^{\ell}+\ell^{-1/n}
\quad{\rm and}\quad
b_j> y_j\ge2^m,$$
which further implies that
$[a_j,b_j]\supsetneqq[2^\ell+\ell^{-1/n}, 2^m]$
and hence \eqref{QBlm} holds true.
This finishes the proof of the above claim.

Now, we rearrange cubes in $\cf_1$.
If there exists some cube $Q\in\cf_1$ such that
$Q^\circ\cap P_1\neq\emptyset$ and
$Q^\circ\cap P_m\neq\emptyset$ for some
$m\in\{2,\dots\}$,
then denote this cube $Q$ by $\widehat{Q}_{(1)}$.
We claim that $\widehat{Q}_{(1)}^\circ\cap P_2\neq\emptyset$.
Indeed, if $m=2$, then the conclusion obviously holds true.
If integer $m>2$, then, using \eqref{QBlm} and a simple
geometrical observation, we obtain
$$\widehat{Q}_{(1)}\supsetneqq B_{1,m}\supsetneqq P_2,$$
which implies that the above claim holds true.
From this claim, the definition of $\widehat{Q}_{(1)}$,
and \eqref{QBlm} again, we deduce that
$$\widehat{Q}_{(1)}\supsetneqq B_{1,2}.$$
Similarly, if there exists another $\widetilde{Q}\in\cf_1$
such that
$\widetilde{Q}^\circ\cap P_1\neq\emptyset$ and
$\widetilde{Q}^\circ\cap P_m\neq\emptyset$ for some
$\widetilde{m}\in\{2,\dots\}$,
then we also have
$$\widetilde{Q}\supsetneqq B_{1,2}$$
and hence
$$\lf(\widetilde{Q}\cap \widehat{Q}_{(1)}\r)\supsetneqq B_{1,2},$$
which contradicts to the fact that
$\widetilde{Q}^\circ\cap \widehat{Q}_{(1)}^\circ=\emptyset$.
Thus, there exists no more than one cube
$\widehat{Q}_{(1)}\in\cf_1$ such that
$\widehat{Q}_{(1)}^\circ\cap P_1\neq\emptyset$ and
$\widehat{Q}_{(1)}^\circ\cap P_{m_1}\neq\emptyset$ for some
$m_1\in\{2,\dots\}$.
If such cube does not exist, then we let
$\widehat{Q}_{(1)}:=\emptyset$.
Moreover, if there exists some cube $Q\in\cf_1\setminus\{\widehat{Q}_{(1)}\}$
such that $Q^\circ\cap P_2\neq\emptyset$ and
$Q^\circ\cap P_m\neq\emptyset$ for some
$m\in\{3,\dots\}$,
then denote this cube $Q$ by $\widehat{Q}_{(2)}$.
Similarly, we know that
there exists no more than one cube
$\widehat{Q}_{(2)}\in\cf_1$ such that
$\widehat{Q}_{(2)}^\circ\cap P_2\neq\emptyset$ and
$\widehat{Q}_{(2)}^\circ\cap P_m\neq\emptyset$ for some
${m_2}\in\{3,\dots\}$.
Repeating the above procedure,
we obtain $\{\widehat{Q}_{(j)}\}_{j\in\nn}=\cf_1$.
Furthermore, we choose a subsequence of
$\{\widehat{Q}_{(j)}\}_{j\in\nn}$, denoted by
$\{\widehat{Q}_{(j_k)}\}_{k=1}^{K}$ with
$$K:=\sharp\cf_1,$$
such that $\widehat{Q}_{(j_k)}\neq\emptyset$ for any $k\in\{1,\,\dots,\,K\}$,
and $j_1<j_2<\cdots$.
Now, for any $k\in\{1,\,\dots,\,K\}$, let
$$Q_{(k)}:=\widehat{Q}_{(j_k)}.$$
Then we have
$\lf\{Q_{(k)}\r\}_{k=1}^{K}
=\lf\{\widehat{Q}_{(j_k)}\r\}_{k=1}^{K}
=\lf\{\widehat{Q}_{(j)}\r\}_{j\in\nn}=\cf_1$.
To sum up, $\{Q_{(k)}\}_{k=1}^{K}$ is the desired rearrangement of $\cf_1$.

For any $k\in\nn$, by a simple geometrical observation,
we find that there exist
$n_k,\ m_k\in\nn$ with
\begin{align}\label{nkmk}
n_k<m_k
\end{align}
such that
$$\lf\{l\in\nn:\ Q_{(k)}^\circ\cap P_l\neq\emptyset \r\}
=:\{n_k,n_k+1,\,\dots,\,m_k\},$$
which, together with \eqref{QBlm}, further implies that
\begin{align}\label{QkB}
Q_{(k)}\supsetneqq B_{n_k,m_k}.
\end{align}
Moreover, we claim that
\begin{align}\label{mono}
n_1< m_1\le n_2< m_2\le \cdots < m_{k-1}\le n_k< m_k\le n_{k+1}<\cdots.
\end{align}
Without loss of generality, we only need to show $m_1\le n_2$.
Indeed, if $m_1> n_2$, then, from \eqref{nkmk}, it follows that
$$\max\{n_1,n_2\}<\min\{m_1,m_2\}.$$
By this and \eqref{QkB}, we conclude that
$$Q_{(1)}\supsetneqq B_{n_1,m_1}
\supset B_{\max\{n_1,n_2\},\min\{m_1,m_2\}}
$$
and
$$
Q_{(2)}\supsetneqq B_{n_2,m_2}
\supset B_{\max\{n_1,n_2\},\min\{m_1,m_2\}},$$
and hence
$Q_{(1)}\cap Q_{(2)}
\supsetneqq B_{\max\{n_1,n_2\},\min\{m_1,m_2\}}$,
which contradicts to the fact that
$Q_{(1)}^\circ\cap Q_{(2)}^\circ=\emptyset$.
This shows that $m_1\le n_2$, and hence
the above claim holds true.

From \eqref{QkB}, \eqref{mono}, $\onepqa\in(1-\frac{p}{q},0)$,
and the disjointness of $\{Q_{(k)}\}_{k=1}^{K}$,
we deduce that
$${\rm I}_1 =  \sum_{k=1}^K |Q_{(k)}|^{\onepqa} \|f\|^p_{\Lq(Q_{(k)})}
\le \sum_{k=1}^K |B_{n_k,m_k}|^{\onepqa}
\left( \sum_{i=n_k}^{m_k}\frac{1}{i}\right)^{\frac{p}{q}}$$
and hence
\begin{align}\label{I1-1}
{\rm I}_1 &\le\sum_{k=1}^K |B_{n_k,m_k}|^{\onepqa}
\left( \sum_{i=n_k}^{m_k}\frac{1}{i}\right)^{\frac{p}{q}}\noz\\
&=\sum_{k=1}^K \lf[2^{m_k}-2^{n_k}-({n_k})^{-\frac{1}{n}}\r]^{n(\onepqa)}
\left( \sum_{i=n_k}^{m_k}\frac{1}{i}\right)^{\frac{p}{q}}\noz\\
&=\sum_{k=1}^K 2^{n_kn(\onepqa)}\lf[2^{m_k-n_k}-1-2^{-n_k}({n_k})^{-\frac{1}{n}}\r]^{n(\onepqa)}
\left( \sum_{i=n_k}^{m_k}\frac{1}{i}\right)^{\frac{p}{q}}\noz\\
&=\sum_{k=1}^K 2^{n_kn(\onepqa)} c_k,
\end{align}
where
$$c_k:=\lf[2^{m_k-n_k}-1-2^{-n_k}({n_k})^{-\frac{1}{n}}\r]^{n(\onepqa)}
\left( \sum_{i=n_k}^{m_k}\frac{1}{i}\right)^{\frac{p}{q}}.$$
For any $k\in\{1,\,\dots,\,K\}$, by $n_k\ge1$,
we have $2^{-n_k}{n_k}^{-\frac{1}{n}}< \frac{1}{2}$,
which, combined with \eqref{nkmk}, further implies that
\[
\begin{aligned}
c_k =& \lf[2^{m_k-n_k}-1-2^{-n_k}({n_k})^{-\frac{1}{n}}\r]^{n(\onepqa)}
\left( \sum_{i=n_k}^{m_k}\frac{1}{i}\right)^{\frac{p}{q}}\\
<& \lf(2^{m_k-n_k}-\frac{3}{2}\r)^{n(\onepqa)} (m_k-n_k)^{\frac{p}{q}}
\le\|g\|_{L^\fz([1,\fz))}<\fz,
\end{aligned}
\]
where
\begin{align}\label{g-domin}
g(t):=\frac{t^\frac{p}{q}}{(2^t-\frac32)^{n(\frac{p}{q}+p\alpha-1)}},
\quad \forall\,t\in[1,\fz).
\end{align}
By this and \eqref{I1-1}, we conclude that
\begin{align}\label{I1}
{\rm I}_1 \le& \sum_{k=1}^K c_k 2^{n_kn(\onepqa)}
\le \|g\|_{L^\fz([1,\fz))} \sum_{k=1}^K 2^{n_kn(\onepqa)}\noz\\
\le& \|g\|_{L^\fz([1,\fz))} \sum_{k=1}^\infty 2^{kn(\onepqa)}
< \infty.
\end{align}
This is a desired estimate of I$_1$.

From \eqref{I1I2}, \eqref{I2}, \eqref{I1}, and the arbitrariness of
$\{Q_i\}_i$, it follows that
\[
\|f\|_{\RM(\rn)}
= \sup\lf[ \sum_i |Q_i|^{\onepqa}\|f\|^p_{\Lq(Q_i)}\r]^{\frac{1}{p}}
< \infty
\]
and hence $f\in \RM(\rn)$.
Therefore, $f\in RM_{p,q,\az}(\rn)\setminus L^{\frac{p}{1-p\az}}(\rn)$,
which completes the proof of Proposition \ref{prop-rn}.
\end{proof}

\begin{remark}\label{rem-rn}
Let $p\in(1,\fz)$, $q\in[1,p)$, and $\az\in(\frac1p-\frac1q,0)$.
It is easy to show that the function $f$,
in the proof of Proposition \ref{prop-rn},
does not belong to the \emph{weak Lebesgue space}
$$L^{\frac{p}{1-p\az},\fz}(\rn):=\lf\{h\ {\rm is\ measurable\ on\ }\rn:\
\|h\|_{L^{\frac{p}{1-p\az},\fz}(\rn)}<\fz\r\},$$
where
$$\|h\|_{L^{\frac{p}{1-p\az},\fz}(\rn)}:=
\sup_{\lz\in(0,\fz)}\lf[\lz\lf|\lf\{x\in\rn:\
|h(x)|>\lz\r\} \r|^{\frac1p-\az}\r].$$
Indeed, for any $\lz\in(0,1)$, by the definition of
$f$, we have
$$\lf|\lf\{x\in\rn:\ |f(x)|>\lz\r\} \r|
=\sum_{\ell\in\nn}|P_\ell|
=\sum_{\ell\in\nn}\frac1\ell=\fz,$$
which implies that $\|f\|_{L^{\frac{p}{1-p\az},\fz}(\rn)}=\fz$
and hence $f\notin L^{\frac{p}{1-p\az},\fz}(\rn)$.
From this and the proof of Proposition \ref{prop-rn},
it follows that $f\in RM_{p,q,\az}(\rn)\setminus L^{\frac{p}{1-p\az},\fz}(\rn)$
and hence
$$RM_{p,q,\az}(\rn)\nsubseteq L^{\frac{p}{1-p\az},\fz}(\rn).$$
On the other hand, observe that the function $f_1$,
in the proof of Theorem \ref{thm-Q12}(i),
belongs to $L^{\frac{p}{1-p\az},\fz}(\rn)\setminus RM_{p,q,\az}(\rn)$,
which implies that
$$L^{\frac{p}{1-p\az},\fz}(\rn)\nsubseteq RM_{p,q,\az}(\rn).$$
Thus, the Riesz--Morrey space and the weak Lebesgue space
do not cover each other.
\end{remark}

Obviously, the construction of $f$ on $\rn$ in Proposition \ref{prop-rn}
is no longer feasible on a given cube $Q_0$ of $\rn$.
To obtain the desired function on $Q_0$, we borrow some ideas
from the proof of \cite[Proposition 3.2]{dhky18}.
In what follows, for any sets $E_1,\ E_2\subset\rn$,
the \emph{distance} between $E_1$ and $E_2$
is defined by setting
$$\dist(E_1,E_2):=\lf\{|x-y|:\ x\in E_1,\ y\in E_2\r\}.$$

\begin{proposition}\label{prop-Q}
Let $p\in(1,\fz)$, $q\in[1,p)$, $\az\in(\frac1p-\frac1q,0)$,
and $Q_0$ be any given cube of $\rn$.
Then there exists an
$f\in RM_{p,q,\az}(Q_0)\setminus L^{\frac{p}{1-p\az}}(Q_0)$.
\end{proposition}
\begin{proof}
Let $p,\ q,$ and $\az$ be as in this proposition.
Without loss of generality
(using the technique of translation and dilation),
we only need to find such an $f$ on
$Q_0:=[-\frac{L_0}2,\frac{L_0}2]^n$
with some given side length $L_0$
determined later.

For any $i\in\zz_+$, let
the length $l_i := 2^{-\frac{(i+1)^2}{2n}}$
and the distance $d_i := 2^{-\frac{(i+1)^2}{2n}}$.
We define a series of collections, $\{\mathcal{L}_k\}_{k\in\zz_+}$,
with $\mathcal{L}_k$ be the collection of disjoint subcubes
of $Q_0$ for any $k\in\zz_+$, step by step as follows.
\begin{enumerate}
\item[{\rm(a)}$_1$]
Let $\Delta:=\{(x_1,\,\dots,\,x_n)\in\rn:\ x_1=\cdots=x_n\}$;
see the dotted line in the figure below.
All new cubes chosen
in the following step should keep their
left and lower vertices as well as their right and upper vertices
lying on $\Delta$.

\item[{\rm(a)}$_2$]
Let $I_0$ be an open cube with side length $l_0$
and centered at $\mathbf{0}$,
and $\mathcal{L}_0 := \{I_0\}$.

\item[{\rm(a)}$_3$]
For any $i\in\nn$, we define $\mathcal{L}_i$ via $\mathcal{L}_{i-1}$.
Precisely, for any $I\in\mathcal{L}_{i-1}$,
we choose two new open cubes $I'$ on the two side of $I$,
satisfying that $\dist(I, I') = \sqrt{n}d_{i-1}$
and the length of $I'$ is $l_{i}$.
These two $I'$ are called to be the \emph{children} of $I$.
Then
$$\mathcal{L}_i:=\lf\{{\rm the\ children\ of\ }I:\ I\in\cl_{i-1} \r\},$$
and it is also easy to show that
\begin{align}\label{sharpIi}
\sharp{\cl_i}=2^i;
\end{align}
see the figure below for $\mathcal{L}_0 = \{I_0\}$,
$\mathcal{L}_1:=\{I_1,I_1'\}$,
and $\cl_2:=\{I_2,I_2',\widetilde{I}_2,\widetilde{I}_2'\}$
(and also \cite[Fig. 1]{dhky18} for a similar construction
on the interval).

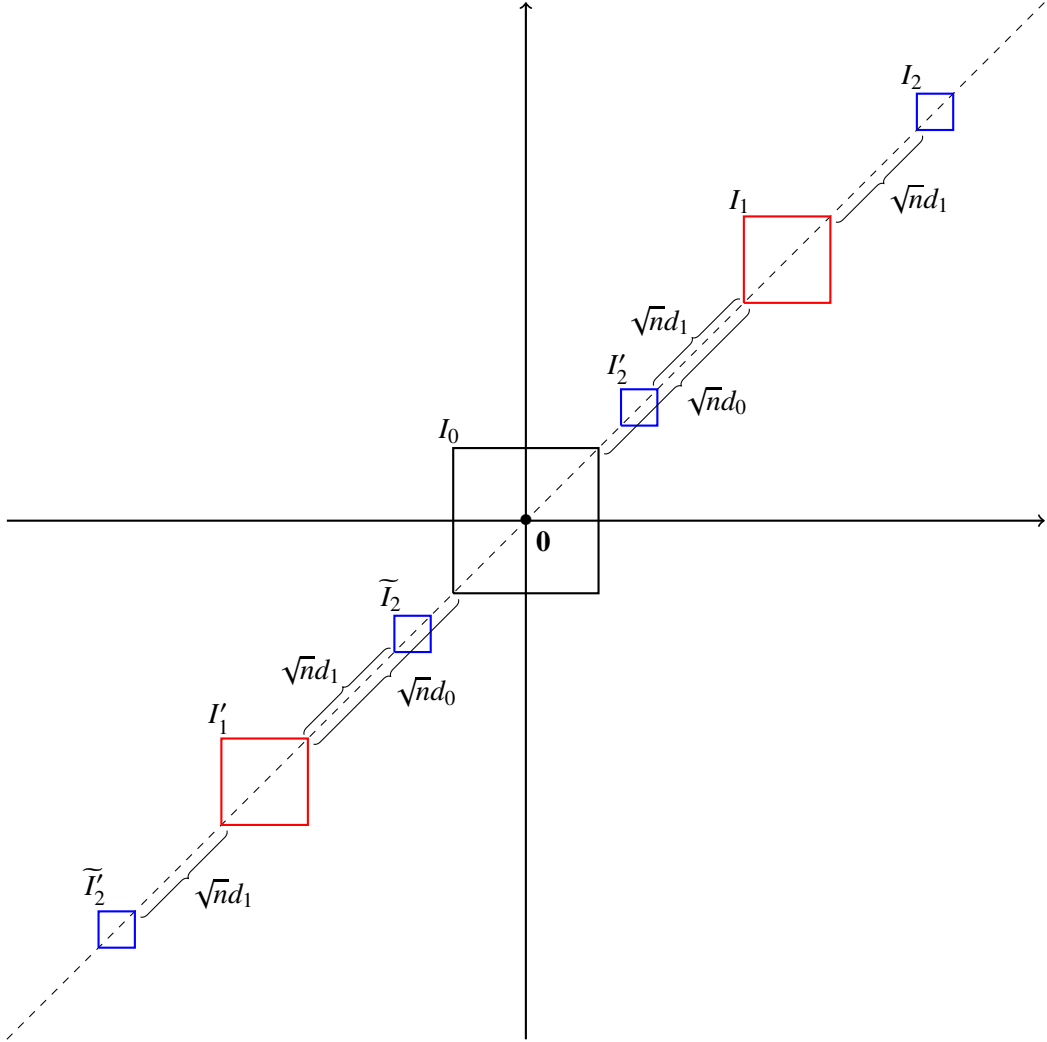
\begin{center}
\begin{tikzpicture} [scale = 2.3]
\draw [thick, ->=5pt] (2,5)--(8,5);
\draw [thick, ->=5pt] (5,2)--(5,8);
\draw [dashed] (2,2)--(8,8);
	
\draw[thick] (4.58,4.58)--(4.58,5.42)--(5.42,5.42)--(5.42,4.58)--(4.58,4.58);
{\red\draw[thick] (3.24,3.24)--(3.24,3.74)--(3.74,3.74)--(3.74,3.24)--(3.24,3.24);}
{\red\draw[thick] (6.26,6.26)--(6.26,6.76)--(6.76,6.76)--(6.76,6.26)--(6.26,6.26);}
{\blue\draw[thick] (2.53,2.53)--(2.53,2.74)--(2.74,2.74)--(2.74,2.53)--(2.53,2.53);}
{\blue\draw[thick] (4.24,4.24)--(4.24,4.45)--(4.45,4.45)--(4.45,4.24)--(4.24,4.24);}
{\blue\draw[thick] (5.55,5.55)--(5.55,5.76)--(5.76,5.76)--(5.76,5.55)--(5.55,5.55);}
{\blue\draw[thick] (7.26,7.26)--(7.26,7.47)--(7.47,7.47)--(7.47,7.26)--(7.26,7.26);}
	
\draw [decoration={brace,mirror,raise=3pt},decorate] (5.42,5.42) -- node[below right] {$\sqrt{n}d_0$} (6.26,6.26);
\draw [decoration={brace,mirror,raise=3pt},decorate] (6.76,6.76) -- node[below right] {$\sqrt{n}d_1$} (7.26,7.26);
\draw [decoration={brace,raise=2pt},decorate] (5.76,5.76) -- node[above left] {$\sqrt{n}d_1$} (6.26,6.26);
\draw [decoration={brace,mirror,raise=3pt},decorate] (3.74,3.74) -- node[below right] {$\sqrt{n}d_0$} (4.58,4.58);
\draw [decoration={brace,raise=2pt},decorate] (3.74,3.74) -- node[above left] {$\sqrt{n}d_1$} (4.24,4.24);
\draw [decoration={brace,mirror,raise=3pt},decorate] (2.74,2.74) -- node[below right] {$\sqrt{n}d_1$} (3.24,3.24);
	
\node [above left] at (4.68,5.39) {$I_0$};
\node [above left] at (6.36,6.72) {$I_1$};
\node [above left] at (3.35,3.71) {$I_1'$};
\node [above left] at (7.355,7.44) {$I_2$};
\node [below right] at (5,5) {$\mathbf{0}$};
\node at (5,5) {$\bullet$};
\node [above left] at (5.655,5.73) {$I_2'$};
\node [above left] at (2.635,2.72) {$\widetilde{I}_2'$};
\node [above left] at (4.345,4.42) {$\widetilde{I}_2$};
\end{tikzpicture}
\end{center}

\item[{\rm(a)}$_4$]
Let $\mathcal{L}:=\bigcup_{i\in\mathbb{Z}_+}\mathcal{L}_i$.
\end{enumerate}

We next show that, under some modifications of $\{d_i\}_{i=0}^{N_0}$
with some given $N_0\in\nn$ determined later,
$\cl$ is a collection of disjoint cubes.
To this end, we introduce the notation of \emph{descendants} as follows.
For any given $I\in \mathcal{L}_i$ with $i\in \mathbb{Z}_+$,
let $\mathcal{J}_{I}^{(i)}:= \{I\}$,
$$\mathcal{J}_{I}^{(j)} := \lf\{I':\
I' {\rm\ is\ one\ of\ the\ children\ of\ }
\widetilde{I}{\rm\ for\ some\ }\widetilde{I}\in\mathcal{J}_{I}^{(j-1)} \r\}$$
for any integer $j\in[i+1,\fz)$. Then
$\mathcal{J}_{I}:=\bigcup_{j=i}^\infty \mathcal{J}_{I}^{(j)}$
is called the \emph{descendant} of $I$.
Notice that, in particular, $\cl$ is just the descendant of $I_0$.

Now, to investigate the disjointness of $\cl$,
for any given $I\in \mathcal{L}_i$ with $i\in \mathbb{Z}_+$,
we define
\[
D_{I} := \sup_{\widetilde{I}\in \mathcal{J}_{I}}
\sup_{x\in\widetilde{I}}\dist(x, I)
\]
which denotes the ``radius'' of the descendant of $I$,
and claim that
\begin{align}\label{DI}
D_{I} = \sqrt{n}\sum_{k=i}^\infty (d_k+l_{k+1}).
\end{align}
Indeed, we first show that \begin{align}\label{DI1}
	D_{I} \ge \sqrt{n}\sum_{k=i}^\infty (d_k+l_{k+1}).
\end{align}
According to the definition of $\mathcal{J}_{I}$
and some geometrical observations, for any integer $j\in(i,\fz)$,
there exists an $\widetilde{I}\in \mathcal{J}_{I}^{(j)}$ such that
$$\dist\lf(I,\widetilde{I}\r) = \sqrt{n}\lf[-l_j+ \sum_{k=i}^{j-1}(d_k+l_{k+1})\r].$$
By this and a geometrical observation,
we find that there exist a point $x\in \widetilde{I}$
such that
$$\sup_{x\in\widetilde{I}}\dist(x, I)
=\sqrt{n} l_j+\dist(I,\widetilde{I})
=\sqrt{n}\sum_{k=i}^{j-1}(d_k+l_{k+1}),$$
which implies that
$D_{I} \ge \sqrt{n}\sum_{k=i}^{j-1} (d_k+l_{k+1})$
for any integer $j\in(i,\fz)$, and hence \eqref{DI1} holds true.
On the other hand, we show that $D_{I} \le \sqrt{n}\sum_{k=i}^\infty (d_k+l_{k+1})$.
Let integer $j\in(i,\fz)$ and $\widetilde{I}\in \mathcal{J}_{I}^{(j)}$.
If $j=i+1$, then we obviously have
$$\sup_{\widetilde{I} \in \mathcal{J}_I^{(j)}} \sup_{x \in \widetilde{I}} \dist(x,I)
=\sqrt{n}\lf(d_i+l_i\r)\le \sqrt{n}\sum_{k=i}^\infty (d_k+l_{k+1}),$$
which is the desired estimate.
If $j\ge i+2$, then, from the definition of $\mathcal{L}$
and some geometrical observations,
it follows that
$$\dist\lf(\widetilde{I},I\r)
\le \sqrt{n}\lf[d_{j-1}+\sum_{k=i}^{j-2} (d_k+l_{k+1})\r],$$
which further implies that
\begin{align*}
\sup_{\widetilde{I} \in \mathcal{J}_I^{(j)}} \sup_{x \in \widetilde{I}} \dist(x,I)
&\le\sup_{\widetilde{I} \in \mathcal{J}_I^{(j)}}
\lf[\sqrt{n}l_j+\dist\lf(\widetilde{I},I\r)\r]\\
&\le \sqrt{n}\sum_{k=i}^{j-1} (d_k+l_{k+1})
\le \sqrt{n}\sum_{k=i}^\infty (d_k+l_{k+1}).
\end{align*}
To sum up, we have
$$D_{I} \le \sqrt{n}\sum_{k=i}^\infty (d_k+l_{k+1}),$$
which, together with \eqref{DI1}, further shows  that
\eqref{DI}, and hence this claim, holds true.

By \eqref{DI}, we find that, for any given $i\in\zz_+$
and any $I\in\cl_i$,
$D_I=\sqrt{n}\sum_{k=i}^\infty (d_k+l_{k+1})$
depends only on $i$, and hence we can define
\begin{align}\label{Di-def}
D_i:=\sqrt{n}\sum_{k=i}^\infty (d_k+l_{k+1}).
\end{align}
Moreover, for any $i\in\mathbb{Z}_+$, we have
\begin{align}\label{Di}
D_i =& \sum_{k=i}^\infty \lf(d_k+l_{k+1}\r)
= \sum_{k=i}^\infty \lf[2^{-\frac{1}{2n}(k+1)^2}+2^{-\frac{1}{2n}(k+2)^2}\r]\noz\\
=& 2^{-\frac{1}{2n}(i+1)^2} \sum_{k=i}^\infty \left\{2^{-\frac{1}{2n}[(k+1)^2-(i+1)^2]}+2^{-\frac{1}{2n}[(k+2)^2-(i+1)^2]}\right\}\noz\\
=& d_i \lf[\sum_{k=i}^\infty 2^{-\frac{1}{2n}(k-i)(k+i+2)}
+ \sum_{k=i+1}^\infty 2^{-\frac{1}{2n}(i+k+2)(k-i)}\r]
< \frac{2^{\frac{1}{2n}}+1}{2^{\frac{1}{2n}}-1}d_i.
\end{align}

As a counterpart of the above $D_I$,
for any $I\in \mathcal{L}_i$ with $i\in\zz_+$, let
\[
\delta_{I} := \inf_{\widehat{I}\in \mathcal{J}_{I}\setminus\{I\}} \dist(I,\widehat{I})
\]
which denotes the closest distance from $I$ to
$\widehat{I} \in \mathcal{J}_{I}\setminus\{I\}$,
and we claim that
\begin{align}\label{dzI}
\delta_I = \max\lf\{0,\sqrt{n}\lf(d_i-D_{i+1}\r)\r\}.
\end{align}
Indeed, if $\sqrt{n}(d_i-D_{i+1})<0$, then,
from some geometrical observations, it follows that
$I$ intersects some cube in the descendant of $I$,
and hence $\delta_I=0$.
Thus, in what follows, we may assume that
$\sqrt{n}(d_i-D_{i+1})\ge0$,
and first show that
\begin{align}\label{dzI1}
\delta_{I} \le \sqrt{n}\lf(d_i-D_{i+1}\r)
\end{align}
Let $j\in(i,\fz)$ be any given integer.
If $j=i+1$, then, for any $\widehat{I} \in \mathcal{J}_{I}^{(j)}$,
$\widehat{I}$ must be one of the children of $I$ and hence
\begin{align}\label{dwII1}
\dist\lf(\widehat{I},I\r)=\sqrt{n} d_i.
\end{align}
If $j\ge i+2$, then, via some geometrical observations, we find that
there exists an $\widehat{I} \in \mathcal{J}_{I}^{(j)}$ satisfying
\begin{align}\label{dwII2}
\dist\lf(\widehat{I},I\r) = \sqrt{n}\lf[d_i- \sum_{k=i+1}^{j-1}(d_k+l_{k+1})\r].
\end{align}
Combining \eqref{dwII1} and \eqref{dwII2},
and taking the infimum over $\widehat{I} \in \mathcal{J}_{I}^{(j)}$
with integer $j\in(i,\fz)$, we obtain
$$\dz_I\le\lim_{j\to\fz}\sqrt{n}\lf[d_i- \sum_{k=i+1}^{j-1}(d_k+l_{k+1})\r]
=\sqrt{n}\lf(d_i-D_{i+1}\r),$$
which shows that \eqref{dzI1} holds true.
On the other hand, we show that $\delta_{I} \ge \sqrt{n}[d_i-D_{i+1}]$.
Let $j\in(i,\fz)$ be any given integer.
The case $j=i+1$ has been discussed in \eqref{dzI1}.
If $j\ge i+2$, then, via some geometrical observations again,
we find that there exists an $\widehat{I}\in\mathcal{J}_{I}^{j}$ satisfying
\begin{align*}
\inf_{\widetilde{I}\in \mathcal{J}_{I}^j} \dist\lf(\widetilde{I},I\r)
&=\dist\lf(\widehat{I},I\r)
=\sqrt{n}\lf[d_i- \sum_{k=i+1}^{j-1}(d_k+l_{k+1})\r]\\
&\ge\sqrt{n}\lf[d_i- \sum_{k=i+1}^{\fz}(d_k+l_{k+1})\r]
=\sqrt{n}\lf(d_i-D_{i+1}\r).
\end{align*}
Taking the infimum over integer $j\in(i,\fz)$,
we obtain
$$\delta_{I} \ge \sqrt{n}\lf(d_i-D_{i+1}\r),$$
which, combined with \eqref{dzI1}, shows that
\eqref{dzI} and hence the above claim hold true.

From \eqref{dzI}, it follows that, for any given $i\in\zz_+$
and any $I\in\cl_i$,
$\dz_I=\sqrt{n}\lf(d_i-D_{i+1}\r)$
depends only on $i$, and hence we can define
\begin{align}\label{dzi-def}
\dz_i:=\max\lf\{0,\sqrt{n}\lf(d_i-D_{i+1}\r)\r\}.
\end{align}
By the definition of $\dz_i$ (namely, $\dz_I$ with $I\in\cl_i$),
we easily deduce that, if $\delta_i>0$ for any $i\in\zz_+$,
then $\mathcal{L}$ is a collection of disjoint cubes.
Therefore, to obtain the disjointness of $\cl$,
it suffices to check whether or not
\begin{align*}
d_i-D_{i+1}>0
\end{align*}
holds true for any $i\in\zz_+$.
Indeed, from \eqref{Di}, it follows that, for any $i\in\zz_+$,
\begin{align*}
d_i-D_{i+1} >& d_i- \frac{2^{\frac{1}{2n}}+1}{2^{\frac{1}{2n}}-1}d_{i+1}\\
=& 2^{-\frac{1}{2n}(i+1)^2} - \frac{2^{\frac{1}{2n}}+1}{2^{\frac{1}{2n}}-1}2^{-\frac{1}{2n}(i+2)^2}\noz\\
=& \lf[1-\frac{2^{\frac{1}{2n}}+1}{2^{\frac{1}{2n}}-1}2^{-\frac{(2i+3)}{2n}}\r]2^{-\frac{1}{2n}(i+1)^2}\noz\\
=& \lf[1-\frac{2^{\frac{1}{2n}}+1}{2^{\frac{1}{2n}}-1}2^{-\frac{(2i+3)}{2n}}\r]d_i.\noz
\end{align*}
Since $2^{-\frac{(2i+3)}{2n}}$ tends to $0$ as $i \to \infty$,
it follows that
there exists an $N_0\in\nn$ such that, for any integer $i\in(N_0,\fz)$,
\begin{align}\label{dzi>0}
\dz_i=\sqrt{n}\lf(d_i-D_{i+1}\r) > \frac12 d_i>0.
\end{align}
Thus, it remains to handle with the case $i\in\{0,\,\dots,\,N_0\}$,
and we modify it step by step as follows.
\begin{enumerate}
\item[{\rm(b)}$_1$]
Choose a $\widetilde{d}_{N_0}\in(0,\fz)$ such that
$\widetilde{d}_{N_0}-D_{N_0+1}>0$,
and replace $d_{N_0}$ by $\widetilde{d}_{N_0}$
in the above (a)$_{3}$.
Then, by \eqref{Di-def}, we find that $D_{N_0}$ becomes
$$\widetilde{D}_{N_0}=\sqrt{n}\lf[\widetilde{d}_{N_0}+l_{N_0+1}+
\sum_{k=N_0+1}^\infty \lf(d_k+l_{k+1}\r)\r]\in(0,\fz),$$
and $\{D_i\}_{i=N_0+1}^\fz$ do not change.
From this and \eqref{dzi-def},
it follows that $\dz_{N_0}$ becomes
$$\widetilde{\dz}_{N_0}=\sqrt{n}\lf(\widetilde{d}_{N_0}-D_{N_0+1}\r)\in(0,\fz),$$
and $\{\dz_i\}_{i=N_0+1}^\fz$ do not change.

\item[{\rm(b)}$_2$]
Choose a $\widetilde{d}_{N_0-1}\in(0,\fz)$ such that
$\widetilde{d}_{N_0-1}-\widetilde{D}_{N_0}>0$.
Replace $d_{N_0-1}$ and $d_{N_0}$, respectively,
by $\widetilde{d}_{N_0-1}$ and $\widetilde{d}_{N_0}$
in the above (a)$_{3}$.
Then, by \eqref{Di-def}, we find that $D_{N_0-1}$ becomes
$$\widetilde{D}_{N_0-1}=\sqrt{n}\lf[\widetilde{d}_{N_0-1}+l_{N_0}+
\widetilde{d}_{N_0}+l_{N_0+1}+
\sum_{k=N_0+1}^\infty \lf(d_k+l_{k+1}\r)\r]\in(0,\fz),$$
and $\{\widetilde{D}_{N_0},D_{N_0+1},\,\dots\}$ do not change.
From this and \eqref{dzi-def},
it follows that $\dz_{N_0-1}$ becomes
$$\widetilde{\dz}_{N_0-1}=\sqrt{n}\lf(\widetilde{d}_{N_0-1}-\widetilde{D}_{N_0}\r)\in(0,\fz),$$
and $\{\widetilde{\dz}_{N_0},\dz_{N_0+1},\,\dots\}$ do not change.

\item[{\rm(b)}$_3$]
Iterate the above procedure until
$d_{0}$ is replaced by some $\widetilde{d}_{0}\in(0,\fz)$.
Then $D_{0}$ becomes
\begin{align}\label{D0}
\widetilde{D}_{0}
=\sqrt{n}\lf[\sum_{k=0}^{N_0} \lf(\widetilde{d}_k+l_{k+1}\r)
+\sum_{k=2}^\infty \lf(d_k+l_{k+1}\r)\r]\in(0,\fz),
\end{align}
and $\{\widetilde{D}_{1},\,\dots,\,\widetilde{D}_{N_0},D_{N_0+1},\,\dots\}$ do not change.
Also, $\dz_{0}$ becomes
$$\widetilde{\dz}_{0}=\sqrt{n}\lf(\widetilde{d}_{0}-\widetilde{D}_{1}\r)\in(0,\fz),$$
and $\{\widetilde{\dz}_{1},\,\dots,\,\widetilde{\dz}_{N_0},\dz_{N_0+1},\,\dots\}$ do not change.
\end{enumerate}
For any $i\in\zz_+$, let
$$\widehat{d}_i:=
\begin{cases}
	\widetilde{d}_i &{\rm if\ }i\in\{0,\,\dots,\, N_0\},\\
	d_i &{\rm if\ }i\in\{N_0+1,\dots\},
\end{cases}$$
$$
\widehat{D}_i:=
\begin{cases}
	\widetilde{D}_i &{\rm if\ }i\in\{0,\,\dots,\, N_0\},\\
	D_i &{\rm if\ }i\in\{N_0+1,\dots\},
\end{cases}$$
and
$$
\widehat{\dz}_i:=
\begin{cases}
	\widetilde{\dz}_i &{\rm if\ }i\in\{0,\,\dots,\, N_0\},\\
	\dz_i &{\rm if\ }i\in\{N_0+1,\dots\}.
\end{cases}$$
So far, we obtain a new family $\widehat{\cl}$,
associated with
$\{l_i,\widehat{d}_i,\widehat{D}_i,\widehat{\dz}_i\}_{i\in\zz_+}$,
which is a collection of disjoint cubes because
\begin{align*}
\widehat{\dz}_i=\widehat{d}_i-\widehat{D}_{i+1}>0
\end{align*}
for any $i\in\zz_+$.
This is the desired family of disjoint cubes.

In the remainder of this proof,
for the simplicity of the presentation,
we remove the hats of $\widehat{\cl},\ \widehat{d}_i,\ \widehat{D}_i$,
and $\widehat{\dz}_i$ for any $i\in\zz_+$.
Thus, this new $\cl$ is a collection of disjoint cubes.
Keep in mind that, from now on,
$\{d_i,D_i,\dz_i\}_{i=0}^{N_0}$
are some new positive numbers.

Now, choose the side length $L_0$ satisfying
$L_0\in(l_0+2\widetilde{D}_0,\fz)$,
where $l_0=2^{-\frac1{2n}}$, and $\widetilde{D}_0$ is as in \eqref{D0}.
Then $Q_0:=[-\frac{L_0}2,\frac{L_0}2]^n$ contains the descendant of $I_0$, that is,
all subcubes in this new $\mathcal{L}=\bigcup_{i=0}^\fz\cl_i$.
Let
\[f:=  \sum_{i=0}^\fz \sum_{I\in\mathcal{L}_i} h_i\one_{I}\]
with the height $h_i := 2^{\frac{1}{2}(\frac{1}{p}-\alpha)i^2}$
for any $i\in\zz_+$.
Then, by \eqref{sharpIi} and $1-\frac{q}{p}+q\alpha\in(0,1-\frac{q}{p})$,
we obtain
\begin{align}\label{Lq}
\int_{Q_0}|f(x)|^q \,dx
&=\sum_{i=0}^{\infty}\sum_{I_i\in\mathcal{L}_i} h_i^q l_i^n
=\sum_{i=0}^{\infty} 2^i 2^{\frac{1}{2}(\frac{q}{p}-q\alpha)i^2} 2^{-\frac{(i+1)^2}{2}}\noz\\
&=\sum_{i=0}^{\infty} 2^{-\frac{1}{2}(1-\frac{q}{p}+q\alpha)i^2-\frac{1}{2}}
< \infty,
\end{align}
and hence $f\in \Lq(Q_0)$.
Here and thereafter,
$h_i^q:=(h_i)^q$ and $l_i^n:=(l_i)^n$.
Similarly, we also have
\begin{align}\label{Lpa}
\int_{Q_0}|f(x)|^{\pa} \,dx
=&\sum_{i=0}^{\infty}  \sum_{I_i\in\mathcal{L}_i} h_i^{\pa} l_i^n
=\sum_{i=0}^{\infty} 2^i 2^{\frac{1}{2}i^2} 2^{-\frac{(i+1)^2}{2}}	
=\sum_{i=0}^{\infty} 2^{-\frac{1}{2}}
=\infty,
\end{align}
and hence $f\notin \Lpa(Q_0)$.
Thus, to prove this proposition, it remains to show that $f\in\RM(Q_0)$.

Let $\{Q_j\}_{j}$ be any given collection of subcubes of $Q_0$
with pairwise disjoint interiors.
For any $i\in\zz_+$, let
$$\mathcal{G}_i := \lf\{Q\in\{Q_j\}_j:\  \exists\,I\in\mathcal{L}_i
{\rm\ such\ that\ }Q^\circ\cap I\neq\emptyset\r\}$$
and
$$\mathcal{F}_i :=
\begin{cases}
\cg_0&{\rm if\ }i=0,\\
\mathcal{G}_i\setminus\bigcup_{k=0}^{i-1}\mathcal{G}_k
&{\rm if\ }i\in\nn.
\end{cases}$$
Moreover, let $\cf:= \bigcup_{i=0}^\infty \mathcal{F}_i$.

Now, it suffices to consider $\cf$
rather than $\{Q_i\}_i$, because,
for any cube $\widetilde{Q}\in\{Q_i\}_i\setminus\cf$,
we have $\|f\|_{\Lq(\widetilde{Q})} = 0$
and hence $\widetilde{Q}$ does not contribute anything to
the Riesz--Morrey norm. For any $i\in\zz_+$, we
define $\mathcal{F}_{i}^{\rm one}$ and $\mathcal{F}_{i}^{\rm more}$,
respectively, as follows:
\begin{enumerate}
\item[{\rm(c)}$_1$]
$\mathcal{F}_i^{\rm one}$ is defined to be the set of all
$Q\in\mathcal{F}_i$ satisfying that there exists only one cube
$I\in\mathcal{L}=\bigcup_{i=0}^\fz\cl_i$
such that $Q^\circ\cap I\neq\emptyset$.
\item[{\rm(c)}$_2$]
$\mathcal{F}_i^{\rm more}$ is defined to be the set of all
$Q\in\mathcal{F}_i$ satisfying that there exist more than one cube
$I\in\mathcal{L}=\bigcup_{i=0}^\fz\cl_i$
such that $Q^\circ\cap I\neq\emptyset$.
\end{enumerate}
Let $\mathcal{F}^{\rm one} := \bigcup_{i=0}^\infty \mathcal{F}_i^{\rm one}$
and $\mathcal{F}^{\rm more} := \cup_{i=0}^\infty \mathcal{F}_i^{\rm more}$.
So, we have
\begin{align}\label{s+l}
&\sum_{i}|Q_i|^{\onepqa} \|f\|^p_{\Lq(Q_i)}\noz\\
&\quad=\sum_{Q\in\mathcal{F}} |Q|^{\onepqa} \|f\|^p_{\Lq(Q)}
=\sum_{Q\in \mathcal{F}^{\rm one}} |Q|^{\onepqa} \|f\|^p_{\Lq(Q)}
+\sum_{Q\in \mathcal{F}^{\rm more}} \cdots\noz\\
&\quad=\mathcal{L}_{\rm one}+ \mathcal{L}_{\rm more},
\end{align}
where
$$\mathcal{L}_{\rm one}
:=\sum_{Q\in \mathcal{F}^{\rm one}}
|Q|^{\onepqa} \|f\|^p_{\Lq(Q)}$$
and
$$\mathcal{L}_{\rm more}
:=\sum_{Q\in \mathcal{F}^{\rm more}}
|Q|^{\onepqa} \|f\|^p_{\Lq(Q)}.$$

We estimate $\mathcal{L}_{\rm one}$ first.
For any given $i\in\zz_+$ and any $Q \in \mathcal{F}_i^{\rm one}$,
there exists an $I \in \mathcal{L}_i$ such that $Q^\circ\cap I\neq\emptyset$,
and hence we can relabel $Q$ as $Q^{(I)}$.
Moreover, for any $I \in \mathcal{L}_i$,
define $\mathcal{F}_{i,I}^{\rm one}$ to be the set of
all above $Q^{(I)}\in\mathcal{F}_{i}^{\rm one}$.
Then, from the definition of $\mathcal{L}_{\rm one}$
and the disjointness of $\{Q_j^\circ\}_j$, we deduce that,
for any $i\in\zz_+$,
$$\sum_{Q\in \mathcal{F}_{i}^{\rm one}} |Q|^{\onepqa} \|f\|^p_{\Lq(Q)}
=\sum_{I\in \mathcal{L}_i}
\sum_{Q\in \mathcal{F}_{i,I}^{\rm one}} |Q|^{\onepqa} \|f\|^p_{\Lq(Q)}.$$
By this, $\onepqa\in(1-\frac{p}{q},0)$, the definition of $f$,
$1-p\az\in(1,\frac{p}{q})$, the definitions of $\{h_i,l_i\}_{i\in\zz_+}$,
and \eqref{sharpIi}, we conclude that
\begin{align}\label{one}
\mathcal{L}_{\rm one}
=&\sum_{Q\in \mathcal{F}^{\rm one}} |Q|^{\onepqa} \|f\|^p_{\Lq(Q)}
=\sum_{i=0}^{\infty}  \sum_{Q\in \mathcal{F}_{i}^{\rm one}} |Q|^{\onepqa} \|f\|^p_{\Lq(Q)}\noz\\
=&  \sum_{i=0}^{\infty}  \sum_{I\in \mathcal{L}_i}
\sum_{Q\in \mathcal{F}_{i,I}^{\rm one}} |Q|^{\onepqa} \|f\|^p_{\Lq(Q)}\noz\\
\le&  \sum_{i=0}^{\infty}  \sum_{I\in \mathcal{L}_i}
\sum_{Q\in \mathcal{F}_{i,I}^{\rm one}} |Q\cap I|^{\onepqa} \|f\|^p_{\Lq(Q\cap I)}\noz\\
=&  \sum_{i=0}^{\infty}  \sum_{I\in \mathcal{L}_i}
\sum_{Q\in \mathcal{F}_{i,I}^{\rm one}} h_i^p |Q\cap I|^{1-p\alpha}\noz\\
=&  \sum_{i=0}^{\infty}  \sum_{I\in \mathcal{L}_i} |I|^{1-p\alpha}
\sum_{Q\in \mathcal{F}_{i,I}^{\rm one}} h_i^p
\lf(\frac{|Q\cap I|}{|I|}\r)^{1-p\alpha}\noz\\
\le&  \sum_{i=0}^{\infty}  \sum_{I\in \mathcal{L}_i} |I|^{1-p\alpha}
\sum_{Q\in \mathcal{F}_{i,I}^{\rm one}} h_i^p \frac{|Q\cap I|}{|I|}\noz\\
\le&  \sum_{i=0}^{\infty}  \sum_{I\in \mathcal{L}_i} h_i^p l_i^{n(1-p\alpha)}
= \sum_{i=0}^{\infty} 2^i 2^{\frac{1}{2}(1-p\alpha)i^2} 2^{-\frac{1}{2}(1-p\alpha)(i+1)^2}\noz\\
=& 2^{-(1-p\alpha)/2} \sum_{i=0}^{\infty} 2^{p\alpha i}
=\frac{2^{(1-p\alpha)/2}}{1-2^{p\alpha}}.
\end{align}
This is a desired estimate of $\mathcal{L}_{\rm one}$.

Next, we estimate $\mathcal{L}_{\rm more}$.
Let $i\in\zz_+$ and $Q\in\mathcal{F}_i^{\rm more}$.
From the definition of $\mathcal{F}_i^{\rm more}$,
it follows that there exist more than one $I\in\mathcal{L}$
such that $Q^\circ\cap I\neq\emptyset$.
This, together with some geometrical observations and
the definition of $\dz_i$, further implies that
$Q$ must contain some cube with side length $l = \frac{\delta_i}{\sqrt{n}}$.
Therefore, we obtain, for any $i\in\zz_+$ and $Q\in\mathcal{F}_i^{\rm more}$,
\begin{align}\label{dzi0}
|Q|\ge \lf(\frac{\delta_i}{\sqrt{n}}\r)^n.
\end{align}
This, together with \eqref{dzi>0}, further implies that,
for any integer $i\in(N_0,\fz)$ and any $Q\in\mathcal{F}_i^{\rm more}$,
we have
\begin{align}\label{dzi2}
|Q|\ge \lf(\frac{\delta_i}{\sqrt{n}}\r)^n>\lf(\frac{d_i}2\r)^n.
\end{align}

Now, we claim that the interior of any given $Q\in\mathcal{F}_i$ with $i\in\nn$
intersects only one cube in $\mathcal{L}_i$.
Indeed, otherwise $Q^\circ$ intersects at least two cubes in $\mathcal{L}_i$.
From this and the definition of $\mathcal{L}$,
it follows that $Q\supset \widetilde{I}$
for some $\widetilde{I}\in\mathcal{L}_{j}$ with $j\in \{0,1,\cdots,i-1\}$,
and hence $Q\in \cg_{i-1}$,
which contradicts to the fact that
$Q\in \cf_{i}=\mathcal{G}_i\setminus\bigcup_{k=0}^{i-1}\mathcal{G}_k$.
Therefore, the above claim holds true.
By this claim, we can relabel $Q$ as $Q^{(I)}$ with some $I\in\mathcal{L}_i$.
Meanwhile, for any $I\in\mathcal{L}_i$,
from this claim, the disjointness of $\{Q_j\}_{j}$,
and some geometrical observations, it follows that
there exist no more than two cubes $Q^{(I)}\in\mathcal{F}_i^{\rm more}$,
denoted by $\mathcal{F}_{i,I}^{\rm more}$,
such that $(Q^{(I)})^\circ\cap I\neq\emptyset$.
This implies that
\begin{align}\label{sharpFm}
\sharp\mathcal{F}_i^{\rm more}
\le2\sharp\mathcal{L}_i=2^{i+1}
\end{align}
and
\begin{align}\label{iIlong}
\sum_{Q\in \mathcal{F}_i^{\rm more}} |Q|^{\onepqa} \|f\|^p_{\Lq(Q)}
= \sum_{I\in\mathcal{L}_i}  \sum_{Q\in \mathcal{F}_{i,I}^{\rm more}}
|Q|^{\onepqa} \|f\|^p_{\Lq(Q)}.
\end{align}
Moreover, by the definition of $\mathcal{L}$ and this claim again,
we find that $Q=Q^{(I)}$ contains (at most) the descendant of $I$.
Using this as well as the definitions of $f$ and
$\{h_i,l_i\}_{i\in\zz_+}$, we have
\begin{align}\label{desce}
\|f\|_{\Lq(Q)}
\le \left( \sum_{k=i}^{\infty} 2^{k-i} h_k^q l_k^n\right)^{\frac{1}{q}}.
\end{align}

Thus, for any integer $i\in(N_0,\fz)$, from \eqref{iIlong},
\eqref{dzi2}, \eqref{desce},
$\onepqa\in(1-\frac{p}{q},0)$,
the definitions of $\{d_i,h_i,l_i\}_{i=N_0+1}^\fz$,
$1-\frac{q}{p}+q\alpha\in(0,1-\frac{q}{p})$, and \eqref{sharpIi},
we deduce that
\begin{align*}
&\sum_{Q\in \mathcal{F}_i^{\rm more}} |Q|^{\onepqa} \|f\|^p_{\Lq(Q)}\\
&\quad=  \sum_{I\in \mathcal{L}_i}  \sum_{Q\in \mathcal{F}_{i,I}^{\rm more}}
|Q|^{\onepqa} \|f\|^p_{\Lq(Q)}\\
&\quad\lesssim  \sum_{I\in \mathcal{L}_i} d_i^{n(\onepqa)}
\left(\sum_{k=i}^{\infty} 2^{k-i} h_k^q l_k^n\right)^{\frac{p}{q}}
\sim 2^i d_i^{n(\onepqa)} \left( \sum_{k=i}^{\infty} 2^{k-i} h_k^q l_k^n\right)^{\frac{p}{q}}\noz\\
&\quad\sim \left[ 2^{\frac{q}{p}i} 2^{-\frac{1}{2}(\frac{q}{p}-q\alpha-1)(i+1)^2}
\sum_{k=i}^{\infty} 2^{k-i} 2^{\frac{1}{2}(\frac{q}{p}-q\alpha)k^2}
2^{-\frac{1}{2}(k+1)^2} \right]^{\frac{p}{q}}\\
&\quad\sim \left[2^{q\alpha i}
\sum_{k=i}^{\infty} 2^{-\frac{1}{2}(1-\frac{q}{p}+q\alpha)(k^2-i^2)} \right]^{\frac{p}{q}}
\lesssim 2^{p\alpha i}
\end{align*}
and hence
\begin{align}\label{long1}
\sum_{i=N_0+1}^\fz  \sum_{Q_i\in \mathcal{F}_i^{\rm more}} |Q|^{\onepqa} \|f\|^p_{\Lq(Q)}
\lesssim&  \sum_{i=N_0+1}^\fz 2^{p\alpha i}< \infty
\end{align}
due to $\az\in(\frac1p-\frac1q,0)$.
Meanwhile, by \eqref{dzi0}, $\onepqa\in(1-\frac{p}{q},0)$,
$\{\dz_i\}_{i=0}^{N_0}\subset (0,\fz)$, \eqref{sharpFm}, and \eqref{Lq},
we conclude that
\begin{align}\label{long2}
&\sum_{i=0}^{N_0}\sum_{Q\in \mathcal{F}_i^{\rm more}} |Q|^{\onepqa} \|f\|^p_{\Lq(Q)}\noz\\
&\quad\le \sum_{i=0}^{N_0}\lf\{2^{i+1}
\lf[\min_{i = \{0,1, \cdots, N_0\}}{\lf(\frac{\delta_i}{\sqrt{n}}\r)^{n(\onepqa)}}\r]
\|f\|^p_{\Lq(Q_0)}\r\} \noz\\
&\quad\ls 2^{N_0}\lf[\min_{i = \{0,1, \cdots, N_0\}}{\delta_i^{n(\onepqa)}}\r]
\|f\|^p_{\Lq(Q_0)}
< \infty.
\end{align}

Using \eqref{long1} and \eqref{long2}, we obtain
$$\mathcal{L}_{\rm more}
=\sum_{Q\in \mathcal{F}^{\rm more}} |Q|^{\onepqa} \|f\|^p_{\Lq(Q)}
=\sum_{i=0}^{\infty}  \sum_{Q\in \mathcal{F}_{i}^{\rm more}}
|Q|^{\onepqa} \|f\|^p_{\Lq(Q)}
<\fz,$$
which, combined with \eqref{one} and \eqref{s+l},
further implies that
$$\sum_{i}|Q_i|^{\onepqa} \|f\|^p_{\Lq(Q_i)}<\fz.$$
From this and the arbitrariness of
$\{Q_i\}_i$, it follows that
\[
\|f\|_{\RM(Q_0)}
= \sup\left[\sum_i |Q_i|^{\onepqa}\|f\|^p_{\Lq(Q_i)}\right]^{\frac{1}{p}}
< \infty
\]
and hence $f\in \RM(Q_0)$.
Therefore, $f\in \RM(Q_0)\setminus L^{\frac{p}{1-p\az}}(Q_0)$
due to \eqref{Lpa},
which completes the proof of Proposition \ref{prop-Q}.
\end{proof}

\begin{remark}\label{rem-Q0}
Let $p\in(1,\fz)$, $q\in[1,p)$, and $\az\in(\frac1p-\frac1q,0)$.
Observe that the function $f_1$,
in the proof of Theorem \ref{thm-Q12}(i) with some dilation and translation,
belongs to $L^{\frac{p}{1-p\az},\fz}(Q_0)\setminus RM_{p,q,\az}(Q_0)$,
which implies that
$$L^{\frac{p}{1-p\az},\fz}(Q_0)\nsubseteq RM_{p,q,\az}(Q_0),$$
where the \emph{weak Lebesgue space} $L^{\frac{p}{1-p\az},\fz}(Q_0)$
is defined as in Remark \ref{rem-rn} with $\rn$ replaced by $Q_0$.
However, the example $f$ in the proof of Proposition \ref{prop-Q} also belongs
to both $L^{\frac{p}{1-p\az},\fz}(Q_0)$ and $RM_{p,q,\az}(Q_0)$,
and hence we can not deduce
\begin{align}\label{openQ}
RM_{p,q,\az}(Q_0)\nsubseteq L^{\frac{p}{1-p\az},\fz}(Q_0)
\end{align}
from this function.
As a counterpart of Remark \ref{rem-rn},
it is interesting to ask whether or not \eqref{openQ}
still holds true. This is still \emph{unclear} so far.
\end{remark}

Based on above three propositions, we immediately complete the
proof of Theorem \ref{thm-Q12}(ii).
\begin{proof}[Proof of Theorem \ref{thm-Q12}(ii)]
It follows directly from Propositions \ref{prop-<}, \ref{prop-rn}, and \ref{prop-Q}.
This finishes the proof of Theorem \ref{thm-Q12}(ii).
\end{proof}

At the end of this article, according to \cite[Theorem 1 and Corollary 1]{tyy21}
and Theorems \ref{thm-Q23} and \ref{thm-Q12},
it is easy to summarize all the classifications
of the Riesz--Morrey space in the following corollary,
and we omit the details here.
\begin{corollary}\label{coro}
\begin{itemize}
\item[{\rm(i)}]
Let $p\in(1,\fz]$ and $q\in[1,p)$. Then
$$RM_{p,q,\az}(\rn)
\begin{cases}
=L^q(\rn) &{\rm if\ } \az=\frac1p-\frac1q,\\
\supsetneqq L^{\frac{p}{1-p\az}}(\rn)&{\rm if\ } \az\in\lf(\frac1p-\frac1q,0\r),\\
=L^p(\rn) &{\rm if\ } \az=0,\\
=\{0\} &{\rm if\ } \az\in\lf(-\fz,\frac1p-\frac1q\r)\cup(0,\fz).
\end{cases}$$
In particular, $RM_{\fz,q,\az}(\rn)=M_{q,\az}(\rn)$ if $\az\in(-\frac1q,0)$.
	
\item[{\rm(ii)}]
Let $p\in[1,\fz]$ and $q\in[p,\fz]$. Then
$$RM_{p,q,\az}(\rn)
\begin{cases}
=L^q(\rn) &{\rm if\ } \az=\frac1p-\frac1q=0,\\
=\{0\} &{\rm if\ } \az=\frac1p-\frac1q\neq0,\\
=\{0\} &{\rm if\ } \az\in\rr\setminus\lf\{\frac1p-\frac1q\r\}.
\end{cases}$$
	
\item[{\rm(iii)}]
Let $p\in(1,\fz]$, $q\in[1,p)$, and
$Q_0$ be any cube of $\rn$. Then
$$RM_{p,q,\az}(Q_0)
\begin{cases}
=L^q(Q_0) &{\rm if\ } \az=\lf(-\fz,\frac1p-\frac1q\r],\\
\supsetneqq L^{\frac{p}{1-p\az}}(Q_0)&{\rm if\ } \az\in\lf(\frac1p-\frac1q,0\r),\\
=L^p(Q_0) &{\rm if\ } \az=0,\\
=\{0\} &{\rm if\ } \az\in(0,\fz).
\end{cases}$$
In particular, $RM_{\fz,q,\az}(Q_0)=M_{q,\az}(Q_0)$ if $\az\in(-\frac1q,0)$.

\item[{\rm(iv)}]
Let $p\in[1,\fz]$, $q\in[p,\fz]$, and
$Q_0$ be any cube of $\rn$. Then
$$RM_{p,q,\az}(Q_0)
\begin{cases}
=L^q(Q_0) &{\rm if\ } \az\in(-\fz,0],\\
=\{0\} &{\rm if\ } \az\in(0,\fz).
\end{cases}$$
\end{itemize}
\end{corollary}

\noindent\textbf{Acknowledgement}.
Zongze Zeng and Jin Tao would like to thank
Yangyang Zhang for some helpful discussions
on Proposition \ref{prop-rn}.

\bigskip

\noindent Zongze Zeng, Jin Tao and Dachun Yang (Corresponding author)

\smallskip

\noindent  Laboratory of Mathematics and Complex Systems
(Ministry of Education of China),
School of Mathematical Sciences, Beijing Normal University,
Beijing 100875, People's Republic of China

\smallskip

\noindent {\it E-mails}: \texttt{zzzeng@mail.bnu.edu.cn} (Z. Zeng)

\noindent\phantom{{\it E-mails:}} \texttt{jintao@mail.bnu.edu.cn} (J. Tao)

\noindent\phantom{{\it E-mails:}} \texttt{dcyang@bnu.edu.cn} (D. Yang)

\bigskip

\noindent Der-Chen Chang

\medskip

\noindent  Department of Mathematics and Statistics, Georgetown University,
Washington D. C. 20057, USA\\
Graduate Institute of Business Adminstration, College of Management,
Fu Jen Catholic University, New Teipei City 242, Taiwan,
Republic of China

\smallskip

\noindent{\it E-mail:}
\texttt{chang@georgetown.edu}

\end{document}